\documentclass{svmult}

\usepackage{graphicx}
\usepackage{amssymb}
\usepackage{amsmath}
\usepackage{epsfig}
\usepackage{enumerate}
\usepackage{float}
\usepackage{mathptmx}
\usepackage{bm}
\usepackage{multirow}

\makeindex             

\begin{document}

\title*{A New Class of Monotone/Convex Rational Fractal Function}
\titlerunning{Rational Fractal Function}
\author{S. K. Katiyar$^\dag$, A. K. B. Chand$^{\dag}$}
\institute{$^\dag$Department of
Mathematics, Indian Institute of Technology Madras, Chennai -
600036, India\\\email{sbhkatiyar@gmail.com, chand@iitm.ac.in}\\
{\bf This work is a copyright material of IIT Madras and part of CHAPTER-3 of Ph.D. Thesis titled SHAPE PRESERVING RATIONAL AND COALESCENCE
FRACTAL INTERPOLATION FUNCTIONS AND
APPROXIMATION BY VARIABLE SCALING FRACTAL
FUNCTIONS by Dr. Saurabh Kumar Katiyar, July 2017.}
}
\maketitle
\abstract{This paper presents a description and analysis of a rational cubic spline FIF (RCSFIF) that has two shape parameters in each subinterval when it is defined implicitly. To be precise, we consider the iterated function system (IFS) with $q_n=\frac{P_n}{Q_n}$, $n \in \mathbb{N}_{N-1}$,
where $P_n(x)$ are cubic polynomials to be determined
through interpolatory conditions of the corresponding FIF and
$Q_n(x)$ are preassigned quadratic polynomials each containing
two free shape/rationality parameters. We establish the convergence
of the proposed RCSFIF $g$ to the original
function $\Phi \in \mathcal{C}^3(I)$ with respect to the uniform
norm. We also provide the sufficient conditions for an automatic selection of the rational IFS parameters to preserve monotonicity and convexity of a prescribed set of data points.
We consider some examples to illustrate the developed fractal
interpolation scheme and its shape preserving aspects.
}
\noindent\textbf{Keywords} Iterated Function System. Fractal Interpolation Functions. Rational cubic fractal functions. Rational cubic interpolation. Constrained Interpolation. Monotonicity. Convexity
\section{Introduction}\label{HALDIASANGEETAsec1}
In some practical situations such as computer aided geometric design (CAGD), computer-aided design (CAD),
computer graphics (CG), scientific data visualization, information sciences, data are arising from complex functions or scientific phenomena. It is often required to generate a smooth function (practical shape-preserving interpolation spline) that interpolates a prescribed set of data and visualize positive, monotone and/or
convex set of data. Monotonicity plays important roles in various scientific problems such as approximation of copulas and quasi copulas in statistics, stress-strain relationship, rate of dissemination of drug in blood, dose-response curve, fuzzy logic, empirical option of pricing models in finance. Convexity plays a vital role in the theory of non-linear programming which arises in engineering and scientific applications such as  optimal control, parameter estimation, design, and approximation of functions. Rational cubic
splines have been successfully replaced the ordinary polynomials without changing the data in shape-preserving surroundings due to the fact that they possess less oscillatory nature, easiness and excellent asymptotic or tension properties. In recent years, a large number of approaches and achievements
have been reported for shape-preserving interpolation methods. Among a substantial amount of references concerning this topic, the reader is referred to (see, for instance, \cite{DQ,DQ1,GD85,FC,SH,SHN,SHH}, and references therein).
\par
Fractal interpolation function introduced by Barnsley \cite{B1} defined through IFS is a modern technique of interpolation that can retain irregularity or smoothness of prescribed data. The traditional nonrecursive interpolants (for instance, polynomial, spline, rational, trigonometric) are about constructing a very smooth function passing through a given data set, sometimes infinitely (piecewise) differentiable in each of the open subintervals determined by the knots. However, in several physical experiments such as financial series, seismic data, bioelectric recordings and Brownian motion, data arise from highly irregular curves and surfaces found in nature, and may not be generated from smooth functions. This served as a motivation for the development of new types of interpolation functions using fractal methodology. To broaden their horizons, some special class of fractal interpolants are introduced and their shape preserving aspects are investigated recently in the literature. As a submissive 
contribution to this goal, Chand and coworkers have initiated the study on shape preserving fractal interpolation and approximation using various families of polynomial and rational IFSs (see, for instance, \cite{CV2,CVN,CNVS}). These shape preserving fractal interpolation schemes are not well explored hitherto. The purpose of this paper is to present a kind of RCSFIF with two families of shape parameters. The associated IFS involves rational function of the form $\frac{P_n}{Q_n}$, $n \in \mathbb{N}_{N-1}$, where $P_n(x)$ are cubic polynomials to be determined through interpolatory conditions of the corresponding FIF and $Q_n(x)$ are preassigned quadratic polynomials each containing two free shape/rationality parameters. The attractor of the rational IFS in the graph of RCSFIF preserves the tension effects. However, the RCSFIF recovers the traditional rational interpolation scheme introduced by Sarfraz et al. \cite{SHH}, when the scaling factor in each subinterval is taken to be zero, corroborating the power 
of this methodology. A convergence analysis establishes an error bound and shows that the order of approximation is $O(h^3)$ accuracy. We provide the sufficient conditions for an automatic selection of the rational IFS parameters to preserve monotonicity and convexity of a prescribed set of data points. To obtain the visually desirable shape, scaling factors and shape parameters can be adjusted by using optimization techniques. The advantage of the proposed RCSFIF is that for prescribed data, one can have an infinite
number of shape preservating interpolants depending on the shape parameters (scaling factors) of the IFS. Therefore, without a doubt, the rational IFSs produce more versatile and flexible class of interpolating functions compared to the traditional non-recursive interpolation methods.
\par
The rest of this paper is organized as follows. In Section \ref{HALDIASANGEETAsec2}, we briefly recall some preliminary
notations and results. In Section \ref{HALDIASANGEETAsec3}, we construct RCSFIF with two family
of shape parameters. An upper bound for the interpolation error of the
developed RCSFIF is obtained and consequently the convergence analysis
is carried out in Section \ref{HALDIASANGEETAsec5}. Sufficient conditions for the proposed
interpolation spline to preserve the shape of the resulting
$\mathcal{C}^1$-RCSFIF is broached in Section \ref{HALDIASANGEETAsec4}.  Finally, illustrating particular cases to reflect the generality of this work
by numerical examples in Section \ref{HALDIASANGEETAsec6} and final comments are given in Section \ref{HALDIASANGEETAsec7}.
\section{Basics of fractal and fractal interpolation function}\label{HALDIASANGEETAsec2}
In this section we introduce the basic terminologies required for our work. For a more extensive treatment, the reader may consult \cite{B1,B2,PRM}.
\par
For $r\in \mathbb{N}$, let  $\mathbb{N}_r$ denote the subset $\{1,2,\dots, r\}$ of $\mathbb{N}$. Consider a set of data points $\{(x_i, y_i) \in \mathbb{R}^2: i \in \mathbb{N}_N\}$ satisfying $ x_1<x_2<\dots<x_N$, $N>2$, be given. Set $I = [x_1,x_N]$, $I_n = [x_n,x_{n+1}]$ for $n \in \mathbb{N}_{N-1}$. Suppose $L_n: I \rightarrow I_n$, $n \in \mathbb{N}_{N-1}$ be contraction homeomorphisms such that
\begin{equation}\label{Chapter1eq1}
L_n(x_1) =x_n,~~L_n(x_N)=x_{n+1}.
\end{equation}
Let $0<r_n<1, n\in \mathbb{N}_{N-1}$, and $X:=I \times \mathbb{R}$. Let $N-1$ continuous mappings $F_n: X \to \mathbb{R}$ be given satisfying:
\begin{equation}\label{Chapter1eq2}
\vert F_n(x,y)-F_n(x,y^*)\vert \leq r_n \vert
 y-y^*\vert,~ F_n(x_1,y_1)=y_n,\ \ F_n(x_N,y_N)=y_{n+1},
\end{equation}
where $(x,y), (x,y^*)\in X$. Define functions $W_n: X \to I_n \times \mathbb{R}, \ W_n(x,y)=\big(L_n(x),F_n(x,y)\big)$
$\forall~ n \in \mathbb{N}_{N-1} $. For the IFS $\mathcal{I}=\{\mathbb{R}^2; W_n:n \in \mathbb{N}_N\}$, Barnsley \cite{B1} presented the following result.\\
\begin{theorem}\cite{B1}\label{Barnsleythm2}
(i) $\exists$ a metric $d^*$ in $\mathbb{R}^2$ for which the IFS $\mathcal{I}$ is hyperbolic and $d^*$ is equivalent to Euclidean metric.\\
(ii) The IFS $\mathcal{I}$ admits a unique attractor $G,$ and $G$ is the graph  of a continuous function
$g:I\to \mathbb{R}$ which obeys $g(x_i)=y_i$\;for $i \in \mathbb{N}_N$.
\end{theorem}
\begin{definition}
The aforementioned function $g$  whose graph is the attractor of an IFS is called a FIF or a self-referential function corresponding to the IFS $\{X;\ W_n:n\in \mathbb{N}_{N-1}\}.$
\end{definition}
The above FIF $g$ is obtained as the fixed point of the Read-Bajraktarevi\'{c} operator $T$ on a complete metric space $(\mathcal{G}, \rho)$ defined as
\begin{equation}\label{Chapter1eq3}
(Tg^*) (x)=  F_n\left(L_n^{-1}(x), g^*\circ L_n^{-1}(x)\right)\;
\forall~ x\in I_n ,\; n\in \mathbb{N}_{N-1},
\end{equation}
where $\rho(g,g^*):= \max\{|g(x)-g^*(x)|: x \in I\}$.
It can be seen that $T$ is a contraction mapping on $(\mathcal{G}, \rho)$ with a contraction factor
$r^*:= \max\{r_n: n\in \mathbb{N}_{N-1}\}<1$. The fixed point of $T$ is the FIF $g$ corresponding to the IFS $\mathcal{I}$. Therefore, $g$ satisfies the functional equation:
\begin{equation}\label{Chapter1eq4}
g(x) = F_n\left(L_n^{-1}(x),g \circ L_n^{-1}(x)\right),\; x\in
I_n,\; n\in \mathbb{N}_{N-1},
\end{equation}
which is equivalent to
\begin{equation}\label{Chapter1eq4a}
g(L_n(x)) = F_n\left((x),g(x)\right),\; x\in
I,\; n\in \mathbb{N}_{N-1}.
\end{equation}
The most extensively studied FIFs in theory and applications  so
far are defined by the mappings:
\begin{equation}\label{Chapter1eq5}
L_n(x)=a_nx+b_n,~  F_n(x,y)=\alpha_n y + q_n(x),~
n\in \mathbb{N}_{N-1}.
\end{equation}
where $|\alpha_n|<1$, the real parameter $\alpha_n$ is called a scaling factor of the transformation $W_n$, and $\alpha = (\alpha_1, \alpha_2, \dots,
\alpha_{N-1})$ is the scale vector corresponding to the IFS. Here $q_n : I \to \mathbb{R}$ are suitable continuous functions so that the maps $F_n$ satisfy conditions in (\ref{Chapter1eq2}). The coefficients $a_n$ and $b_n$ of the affine maps $L_n$ are determined through the conditions given in (\ref{Chapter1eq1}) as
$$a_n=\frac{x_{n+1}-x_n}{x_N-x_1},~~~~ b_n=\frac{x_n
x_N-x_{n+1}x_1}{x_N-x_1}.$$
\subsection{Differentiable Fractal Interpolation Functions}\label{DiffeFIFs}
\noindent For a prescribed data set, a FIF with
$\mathcal{C}^k$-continuity is obtained as the fixed point of IFS
(\ref{Chapter1eq5}), where the scaling factors $\alpha_n$ and the
functions $q_n$ are chosen according to the following theorem.
\begin{theorem}\cite{B2}\label{BH}
Let $\{(x_i,y_i):i \in \mathbb{N}_N\}$ be a given data set with
strictly increasing abscissae. Let $L_n(x) = a_n x + b_n$
satisfies (\ref{Chapter1eq1}) and $F_n(x,y)=\alpha_n y + q_n(x)$
obeys (\ref{Chapter1eq2}) for $n\in \mathbb{N}_{N-1}$. Suppose that for
some integer $k\geq0$, $|\alpha_n|< a_n^k$ and $q_n\in \mathcal{C}^k(I)$,
$n\in \mathbb{N}_{ N-1}$. Let
$$F_{n,p}(x,y)=\frac{\alpha_n y+ q_n^{(p)}(x)}{a_n^p}, \;\; y_{1,p}
= \frac
{q_1^{(p)}(x_1)}{a_1^p-\alpha_1},\:y_{N,p}=\frac{q_{N-1}^{(p)}(x_N)}{a_{N-1}^p-\alpha_{N-1}},
~p\in \mathbb{N}_k.$$
If $ F_{n-1,p}(x_N, y_{N,p})=F_{n,p}(x_1,y_{1,p})$  for $n =2,3,
\dots,N-1$ and  $p \in \mathbb{N}_k$, then the IFS $\big\{I \times
\mathbb{R};\big(L_n(x),F_n(x,y)\big):n\in \mathbb{N}_{N-1}\big\}$
determines a FIF $g\in \mathcal{C}^k(I)$. Further, $g^{(p)}$ is the FIF
determined by $\big\{I \times
\mathbb{R};\big(L_n(x),F_{n,p}(x,y)\big):n \in \mathbb{N}_{N-1}\big\}$
for $ p \in \mathbb{N}_k$.
\end{theorem}
To get a rational  FIF with $\mathcal{C}^{k}$-continuity, $q_n(x)$ is taken as
$\frac{P_n(x)}{Q_n(x)}$, where $P_n(x)$, $Q_n(x)$
are suitably chosen polynomials in $x$ of degree $M, N$ respectively, and
$Q_n(x) \neq 0 $ for every $ x\in [x_1, x_N]$. Then using condition of Theorem \ref{BH}, the existence of smooth rational FIF is proposed in \cite{CVN}. This completes our preparations for the current study, and we are now ready for our main section.
\section{$\mathcal{C}^1$-RCSFIF with Two-Families of Shape Parameters}\label{HALDIASANGEETAsec3}
Let $\{(x_i,y_i,d_i)\in\mathbb{R}^3: i\in \mathbb{N}_N\}$, $x_1<x_2<\dots<x_N$, be a given set of Hermite data points. The desired RCSFIF with two families of shape parameters can be obtained by the IFS given in (\ref{Chapter1eq5}) with
$$q_n(x)\equiv
q_n^*(\theta)=\frac{U_n (1-\theta)^3 + V_n (1-\theta)^2  \theta+ W_n (1-\theta) \theta^2  + Z_n \theta^3 }{u_n +v_n \theta(1-\theta)},\:
\theta = \frac{x-x_1}{x_N-x_1}.$$
With this special choice of $q_n(x)$, the Read-Bajraktarevi\'{c} operator $T$ (cf. (\ref{Chapter1eq3})) has a unique fixed point $g \in \mathcal{G}$,  which satisfies
\begin{equation}\label{HALDIASANGEETAeq6}
\begin{split}
g\big(L_n(x)\big) &= F_n\big(x, g(x)\big)= \alpha_n g(x)+q_n(x),\\
 &= \alpha_n g(x) + \frac{U_n (1-\theta)^3 + V_n (1-\theta)^2  \theta+ W_n (1-\theta) \theta^2  + Z_n \theta^3 }{u_n +v_n \theta(1-\theta)}.
\end{split}
\end{equation}
The conditions $F_n(x_1,y_1)=y_n$, $F_n(x_N,y_N)=y_{n+1}$ can be
reformulated as the interpolation conditions $g(x_n)=y_n$,
$g(x_{n+1})=y_{n+1}$, $ n\in \mathbb{N}_{N-1}$. The interpolatory conditions determine the coefficients $U_n$ and $Z_n$ as follows.
Substituting $x = x_1$ in (\ref{HALDIASANGEETAeq6}), we get
\begin{equation*}
g\big(L_n(x_1)\big)  = \alpha_n  g(x_1) + \frac{U_n}{u_n}
\implies  y_n  = \alpha_n y_1 + \frac{U_n}{u_n}  \implies  U_n
=u_n(y_n-\alpha_n y_1).
\end{equation*}
Similarly, taking $x = x_N$ in (\ref{HALDIASANGEETAeq6}) we obtain $ Z_n  =
u_n(y_{n+1}-\alpha_n y_N)$. \\Now we make $g\in \mathcal{C}^1(I)$  by imposing the conditions prescribed in Theorem \ref{BH}.\\
By hypothesis, $|\alpha_n | \le \kappa a_n$, $ n\in \mathbb{N}_{N-1}$, where $0\le \kappa < 1$. We also have $q_n \in \mathcal{C}^{1}(I)$. Adhering to the notation of Theorem \ref{BH}, for $ n\in \mathbb{N}_{N-1}$, we let\\
\begin{equation*}
\begin{split}
F_{n,1}(x, y)&= \frac{\alpha_n  y+ q_n^{(1)}(x)}{a_n},\\
y_{1,1}&=d_1,\; y_{N,1}=d_N,\; F_{n,1}(x_1, d_1)=d_n,\;
F_{n,1}(x_N,d_N)=d_{n+1}.
\end{split}
\end{equation*}
Then by Theorem \ref{BH}, the FIF $g \in
\mathcal{C}^1(I)$. Further, $g^{(1)}$ is the fractal function
determined by the IFS $\mathcal {I}^*\equiv\big\{\mathbb{R}^2;\big(L_n(x),
F_{n,1}(x,y)\big): n \in \mathbb{N}_{N-1}\big\}$. Consider $\mathcal{G^*} := \{h^*\in \mathcal{C}(I): h^*(x_1)=d_1~\text{and}~h^*(x_N)=d_N\}$ endowed
with the uniform metric. The IFS $\mathcal{I}^*$ induces a
contraction map $T^*: \mathcal{G^*} \rightarrow \mathcal{G^*}$ defined
by $(T^*g^*)\big(L_n(x)\big)=F_{n,1}\big(x, g^*(x)\big),\; x \in I.$
The fixed point of $T^*$ is $g^{(1)}$. Consequently, $g^{(1)}$
satisfies the functional equation:
\begin{equation}\label{HALDIASANGEETAeq7}
g^{(1)}\big(L_n(x)\big) = F_{n,1}\big(x, g^{(1)}(x)\big)
 = \frac{\alpha_n  g^{(1)}(x) + q_n^{(1)}(x)}{a_n}.
\end{equation}
The  conditions $F_{n,1}(x_1,d_1)=d_n$ and
$F_{n,1}(x_N,d_N)=d_{n+1}$ can be reformulated as the
interpolation conditions for the derivative:
$g^{(1)}(x_n)=d_n$ and $g^{(1)}(x_{n+1})=d_{n+1}$, $ n\in \mathbb{N}_{N-1}$.
Applying $x = x_1$ in (\ref{HALDIASANGEETAeq7}), we obtain
\begin{equation*}
\begin{split}
g^{(1)}(L_n(x_1)) &  = \frac{\alpha_n}{a_n}  g^{(1)}(x_1) + \frac{u_nV_n - (3u_n+v_n) U_n}{u_n^2 h_n},\\
\implies   V_n & =  (3u_n+v_n)(y_n-\alpha_ny_1)+u_n h_nd_n-\alpha_n u_n (x_N-x_1) d_1.
\end{split}
\end{equation*}
Similarly, the substitution  $x = x_N$ in (\ref{HALDIASANGEETAeq7}) yields
\begin{center}
$W_n  = (3u_n+v_n)(y_{n+1}-\alpha_ny_N)-u_n h_nd_{n+1}+\alpha_n u_n (x_N-x_1) d_N$.
\end{center}
These values of $U_n, V_n, W_n$, and $Z_n$ reformulate the  desired $\mathcal{C}^1$-rational cubic spline FIF (\ref{HALDIASANGEETAeq6}) to the following:
\begin{equation}\label{HALDIASANGEETAeq8}
g\big(L_n(x)\big) = \alpha_n g(x) + \frac{P_n(x)} {Q_n(x)},
\end{equation}
$P_n(x)\equiv P_n^* (\theta) =   u_n(y_n-\alpha_n y_1) (1-\theta)^3 + \{(3u_n+v_n)(y_n-\alpha_ny_1)+u_n h_nd_n-\alpha_n u_n (x_N-x_1) d_1\} (1-\theta)^2  \theta+
 \{(3u_n+v_n)(y_{n+1}-\alpha_ny_N)-u_n h_nd_{n+1}+\alpha_n u_n (x_N-x_1) d_N\} (1-\theta) \theta^2  + u_n(y_{n+1}-\alpha_n y_N) \theta^3,$\\
$Q_n(x)\equiv Q_n^*(\theta) = u_n +v_n \theta(1-\theta),~ \theta = \frac{x-x_1}{x_N - x_1} $.\\\\
Since the FIF $g$ in (\ref{HALDIASANGEETAeq8}) is derived as the
fixed point of  $T$, it is unique for a fixed choice of
the scaling factors and the shape parameters.
\begin{remark}\label{HALDIASANGEETArem1}(Interval tension property) Let $ \triangle_n = \dfrac{y_{n+1} - y_n}{h_n}$.  (\ref{HALDIASANGEETAeq8}) can be expressed as
\begin{eqnarray}\label{HALDIASANGEETAeq8a}
g(L_n(x)) = &\alpha_n g(x) + (y_n -\alpha_n y_1) (1-\theta)+
(y_{n+1}-\alpha_n y_N)\theta \\& +  \dfrac{u_nh_n \theta (1-\theta)\big [(2\theta-1)\triangle_n^*
+ (1-\theta)d^*_n-\theta d^*_{n+1}\big ]}{Q_n(\theta)} \nonumber,
\end{eqnarray}
where $d^*_n = d_n -  \frac {\alpha_n d_1}{a_n},~
d^*_{n+1} = d_{n+1} -  \frac {\alpha_n d_N}{a_n},~
\triangle^*_n = \triangle_n - \alpha_n \frac{y_N -y_1}{h_n}$.
When $v_n \rightarrow \infty$ in (\ref{HALDIASANGEETAeq8a}),
$g$ converges to the following affine FIF :
\begin{equation}\label{HALDIASANGEETAeq8b}
g(L_n(x)) = \alpha_n g(x) + (y_n -\alpha_n y_1) (1-\theta)
+ (y_{n+1}-\alpha_n y_N)\theta.
\end{equation}
Again if $\alpha_n \rightarrow 0^+$ with $v_n\rightarrow \infty,$
then the rational cubic FIF modifies to  the classical affine interpolant. Hence,
the shape parameter $v_n$ has a vital influence on the graphical
display of data while $u_n$ can assume any positive value. The increase in
the value of parameter $v_n$ in $[x_n,x_{n+1}]$ transforms the rational cubic
functions to the straight line $y_n(1-\theta)+y_{n+1}\theta$.
\end{remark}
\begin{remark}\label{HALDIASANGEETArem2} If $\alpha_n =0$, $ n\in \mathbb{N}_{N-1},$ then the resulting  RCSFIF coincides with the piecewise defined nonrecursive classical  rational cubic interpolant $C$ as
\begin{equation}\label{HALDIASANGEETAeq9}
g(L_n(x)) =\frac{P_n^* (\theta)}{Q_n^* (\theta)},
\end{equation}
where
$P_n^* (\theta) =  u_n y_n(1-\theta)^3+[(3u_n+v_n)y_n+u_n h_n d_n]
(1-\theta)^2 \theta+[(3u_n+v_n)y_{n+1}-u_n h_n d_{n+1}](1-\theta) \theta^2+u_n y_{n+1}\theta^3$,~
$Q_n^*(\theta) = u_n +v_n \theta(1-\theta)$.
Since $\frac {L_n^{-1}(x) -x_1}{x_N-x_1}= \frac{x-x_n}{h_n}=\rho$, from (\ref{HALDIASANGEETAeq9}), for $x \in I_n=[x_n, x_{n+1}]$, we have
\begin{equation}\label{HALDIASANGEETAeq10}
g(x) =  \frac{P_n^*(\rho)}{Q_n^*(\rho)} \equiv C_n(x)~ (say).
\end{equation}
where  $\rho$ is a localized variable. The
rational cubic spline $C\in \mathcal{C}^1(I)$ is defined by
$C\big|_{I_n}=C_n$, $n \in \mathbb{N}_{N-1}$. This illustrates that if we let  $\alpha_n\rightarrow 0$, then the graph of our rational cubic FIF on $[x_n,x_{n+1}]$ approaches the graph of the classical rational cubic interpolant described by Sarfraz and Hussain \cite{SHH}.
\end{remark}
\begin{remark}\label{HALDIASANGEETArem3} It is interesting to note that when
$u_n = 1, v_n = 0 ~ \text{and}
~|\alpha_n|\le \kappa a_n $ for $ n\in \mathbb{N}_{N-1}$, $\kappa \in (0,1)$, in (\ref{HALDIASANGEETAeq8}) then the resulting RCSFIF coincides with the $\mathcal{C}^1$-cubic Hermite FIF \cite{CV2}. If we take $ u_n =1, v_n = 0 ~ \text{and}
~\alpha_n=0$, we obtain for $ x \in [x_n,x_{n+1}]$,~
$g(x) = (2 \theta^3-3\theta^2+1) y_n + (\theta^3-2\theta^2+\theta)h_n d_n+(-2 \theta^3+3\theta^2)y_{n+1}+(\theta^3-\theta^2)h_n d_{n+1}.$
Hence $g$ recovers the classical piecewise $\mathcal{C}^1$-cubic Hermite interpolant over $I$.
\end{remark}
\section{Convergence  Analysis of RCSFIFs}\label{HALDIASANGEETAsec5}
In this section, the uniform error bound for a RCSFIF $g$ is obtained from the Hermite data  $\{(x_i,y_i,d_i): i \in \mathbb{N}_N\}$ satisfying $x_1<x_2<\dots<x_N$, being interpolated and generated from a function $\Phi \in\mathcal{C}^3(I)$. By using $\|\Phi- g\|_{\infty} \le \|\Phi - C\|_{\infty} + \|C - g\|_{\infty}$, we will derive the convergence of $g$ to the original function $\Phi$ using the convergence results for its classical counterpart $C$ and the uniform distance between $g$ and $C$. The first summand in the above inequality is obtained from Theorem 7.1 of \cite{SHH} as $\|\Phi-C\|_\infty \leq \frac{1}{2}\|\Phi^{(3)}\|_\infty \underset{1\leq i\leq {N-1}}\max \{h_n^3 c_n\}$, for some suitable constant $c_n$ independent of $h_n$. The rightmost summand is obtained by using the definition of the Read-Bajraktarevi\'{c} operators for which $g$ is a fixed point and by applying the Mean Value Theorem. To make our presentation simple, we introduce the following notations: $|y|_{\infty} = \max \{|y_n| : n 
\in \mathbb{N}_N\} $, $ |d|_{\infty} = \max \{|d_n| : n \in \mathbb{N}_N \} $, $ |u|_{\infty} = \max \{
|u_n| : n \in \mathbb{N}_{N-1} \}$, $ |v|_{\infty} = \max \{
|v_n| : n \in \mathbb{N}_{N-1} \}$, $|\alpha|_\infty=\max\{|\alpha_n|: n \in \mathbb{N}_{N-1}\}$, $ h = \max \{h_n : n \in \mathbb{N}_{N-1} \} $. The proof is just consequent upon strictly routine matter of simple calculations.
\begin{theorem}\label{perturerror}
Let $\Phi\in \mathcal{C}^3(I)$ be the original function, $g$ be the RCSFIF for $\Phi$ with respect to the
interpolation data $\{(x_i,y_i,d_i) : i \in \mathbb{N}_N\}$. Let the function $q_n$ involved in the IFS generating the FIF $g$ satisfies $\big|\dfrac{\partial q_n(\tau_n,u_n,v_n, \rho)}{\partial
\alpha_n} \big| \le K_0 $ for $|\tau_n| \in (0, a_n)$, all
$n \in \mathbb{N}_{N-1}$, and for some real constant $K_0$. Then,
\begin{eqnarray*}
\|\Phi-g\|_\infty &\leq \frac{1}{2}\|\Phi^{(3)}\|_\infty h^3 c +\frac{|\alpha|_\infty}{s(1-|\alpha|_\infty)}\Big\{|u|_\infty M+
\frac{1}{4}\big[(3|u|_\infty+|v|_\infty)M+|u|_\infty\times\\& (h
|d|_\infty+|I|\max\{|d_1|,|d_N|\})\big]\Big\},
\end{eqnarray*}
where $M=|y|_\infty+ \max\{|y_1|,|y_N|\}$,   $ s = \min\{ s_n : n \in \mathbb{N}_{N-1} \}$ with $s_n =u_n+\frac{1}{4}v_n$, $ |u|_{\infty} = \max \{
|u_n| : n \in \mathbb{N}_{N-1} \}$, $ |v|_{\infty} = \max \{
|v_n| : n \in \mathbb{N}_{N-1} \}$.
\end{theorem}
\begin{proof}
Let $g$ and $C$, respectively, be the rational cubic spline FIF and the traditional nonrecursive cubic interpolant to the data $\{(x_i,\Phi(x_i)) : i \in \mathbb{N}_N\}$. By the triangle inequality
\begin{equation}\label{HALDIASANGEETAeq10a}
\|\Phi-g\|_{\infty} \le \|\Phi - C\|_{\infty} + \|C - g\|_{\infty}.
\end{equation}
We obtain rightmost summand in  (\ref{HALDIASANGEETAeq10a}) by the definition of the Read-Bajraktarevi\'{c} operators for which
$g$ is a fixed point and by applying the Mean Value Theorem. For a prescribed data set and  $\alpha_i$
satisfying $ |\alpha_n| \le  a_n,  n \in \mathbb{N}_{N-1}$, the RCSFIF $g \in$ $\mathcal{C}^1(I)$ is the fixed point of
the Read-Bajraktarevi\'{c} operator $T_{\alpha}$:
\begin{equation}\label{HALDIASANGEETAeq10c}
  (T_{\alpha}g)(x) = \alpha_n g\big(L_n^{-1}(x)\big) + q_n (\alpha_n,u_n,v_n, \phi),
\end{equation}
where $ q_n(\alpha_n,u_n,v_n, \phi) = \dfrac{P_n(\alpha_n,u_n,v_n,\phi)}
{Q_n(u_n,v_n,\phi)}$, $\phi = \frac{L_n^{-1}(x) - x_1}{x_N -
x_1}=\frac{x - x_n}{h_n}, x \in [x_n, x_{n+1}]$, $\; n \in \mathbb{N}_{N-1},$
with $P_n$ and $Q_n$  as in (\ref{HALDIASANGEETAeq8}). Note that the subscript
$\alpha$ is used to emphasize the dependence of the map $T$ on the
scale vector $\alpha$. The coefficients of the rational function
$q_n$ depend on the scaling factor $\alpha_n$ and the shape
parameter $u_n,v_n$, and hence $q_n$ can be thought of as a function
of $\alpha_n$, $u_n,v_n$, and $\phi$. The interpolants $g$ and $C$ are
the fixed points of $T_\alpha$ with $\alpha \ne {\bf 0}$ and $
\alpha = {\bf 0} $ respectively.
\begin{equation*}
\begin{split}
|T_{\alpha} g(x) - T_{\alpha} C(x)| &=  \Big|\big \{ \alpha_n g(L_n^{-1}(x)) + q_n (\alpha_n, u_n,v_n,\phi )\big\} - \big\{ \alpha_n C(L_n^{-1}(x))\\&~~~~+ q_n (\alpha_n, u_n,v_n, \phi )\big\} \Big|, \\
&\le |\alpha|_{\infty} \|g - C \|_{\infty}.
\end{split}
\end{equation*}
From the above inequality we deduce that $T_\alpha$ is a contraction:
\begin{equation}\label{HALDIASANGEETAeq10d}
\|T_{\alpha}g- T_{\alpha} C \|_{\infty} \le  |\alpha|_{\infty}
\|g- C \|_{\infty}.
\end{equation}
Let $ x \in [x_n, x_{n+1}]$  and  $\alpha \ne {\bf 0}$. Using
(\ref{HALDIASANGEETAeq10c}) and the Mean Value Theorem:
\begin{equation*}
\begin{split}
 |T_{\alpha} C(x) - T_{\bf 0} C(x) |  & = \Big|\big \{ \alpha_n  C(L_n^{-1}(x)) + q_n (\alpha_n, u_n,v_n, \phi )\big\} -  q_n (0, u_n,v_n,\phi)\Big|, \\
& \le | \alpha_n| \|C \|_{\infty} + | \alpha_n|  \Big |
\frac{\partial q_n(\tau_n,u_n,v_n,\phi)}{\partial \alpha_n} \Big |,\\
&\le |\alpha_n| ( \|C \|_{\infty} + K_0).
\end{split}
\end{equation*}
Consequently,
\begin{equation}\label{HALDIASANGEETAeq10e}
 \|T_{\alpha} C- T_{\bf 0} C \|_{\infty} \le |\alpha|_{\infty}  ( \|C \|_{\infty} +
 K_0).
\end{equation}
Using  (\ref{HALDIASANGEETAeq10d}) and (\ref{HALDIASANGEETAeq10e}), we obtain
\begin{equation*}
\begin{split}
 \|g - C \|_{\infty} = \| T_{\alpha} g - T_{\bf 0} C \|_{\infty} &  \le  \| T_{\alpha}g - T_{\alpha} C \|_{\infty} + \| T_{\alpha} C - T_{\bf 0} C \|_{\infty}, \\
& \le |\alpha|_{\infty} \| g - C \|_{\infty} + |\alpha|_{\infty} (
\|C \|_{\infty} +  K_0).
\end{split}
\end{equation*}
From the above inequality we can easily get
\begin{equation}\label{HALDIASANGEETAeq10f}
 \|g - C \|_{\infty} \le  \frac { |\alpha|_{\infty} (  \|C \|_{\infty} +  K_0)}{ 1 -|\alpha|_{\infty}
 }.
\end{equation}
Now, we find an upper bound for
$ \|C \|_{\infty}$ and estimate $K_0$, if not optimally, at least
practically. From (\ref{HALDIASANGEETAeq10}), for $ x \in I$,
$$|C(x)| \le \frac {\max \{|P_n^{**}(\rho) | : n \in \mathbb{N}_{N-1} , 0 \le \rho
\le 1 \}}{\min \{| Q_n^*(\rho)| : n \in \mathbb{N}_{N-1}, 0 \le \rho \le 1\}},$$
where $P_n^{**} (\rho)$ is the numerator in (\ref{HALDIASANGEETAeq10}). Using
the extremum calculations of polynomials, {\small
\begin{equation*}
\begin{split}
&|P_n^{**}(\rho)| \le |u_n|| y_n|  (1-\rho)^3  + \{(3|u_n|+ |v_n|)| y_n| + |u_n|h_n |d_n|\} \rho (1- \rho)^2 \\& ~~~~~~~~~+ \{ (3|u_n|+ |v_n|) | y_{n+1}| + |u_n|h_n | d_{n+1}|) \rho^2 (1-\rho) +|u_n|| y_{n+1}| \rho^3,\\
&\implies \underset{ n \in \mathbb{N}_{N-1}, \rho \in [0,1]} \max |P_n^{**}(\rho)
| \le |u|_{\infty} |y|_{\infty} + \frac{1}{4} \{(3|u|_{\infty}+ |v|_{\infty}) |y|_{\infty} +
|u|_{\infty} h |d|_{\infty})\} ,
 \end{split}
\end{equation*} }
and $|Q_n^* (\rho)| = Q_n^* (\rho)  \ge s_n$. Therefore,\\
$$\|C
\|_{\infty} \le \dfrac { |u|_{\infty} |y|_{\infty} + \frac{1}{4} \{(3|u|_{\infty}+ |v|_{\infty}) |y|_{\infty} +
|u|_{\infty} h |d|_{\infty})\} } { \min\{ s_n : n \in \mathbb{N}_{N-1} \}
 }.$$
From  (\ref{HALDIASANGEETAeq8}), for $x \in [x_n, x_{n+1}]$,
$\dfrac{\partial q_n(\alpha_n,u_n,v_n, \rho)}{\partial \alpha_n} = \dfrac{
\tilde{P_n}(u_n,v_n,\rho)}
 {Q_n^*
 (u_n,v_n,\rho)},$ where
\begin{equation*}
\begin{split}
\tilde{P_n}(u_n,v_n,\rho) = & -\{u_n y_1 (1-\rho)^3 + \{(3u_n+ v_n) y_1+ u_n(x_N
- x_1) d_1\}  \rho (1- \rho)^2\\& + \{(3u_n+ v_n) y_N -  u_n  (x_N - x_1)d_N\}
\rho^2 (1-\rho) + u_n y_N \rho^3 \}.
\end{split}
\end{equation*}
Using similar extremum
calculations,
\begin{equation*}
\begin{split}
&\Big | \dfrac{\partial q_n(.,u_n,v_n,\rho)}{\partial \alpha_n} \Big |  \le  K_0, \text{where}
\\&K_0=\frac{[|u|_{\infty}+ \frac{1}{4}(3|u|_{\infty}+|v|_{\infty})]\max\{|y_1|, |y_N|\}+\frac{1}{4}|u|_{\infty}|I|\max\{ |d_1|,|d_N|\}}{\min\{ c_n : n \in \mathbb{N}_{N-1} \}},
\end{split}
\end{equation*}
and $ |I| = x_N - x_1$. Now (\ref{HALDIASANGEETAeq10f})  coupled with $\|C \|_{\infty}$ and  $K_0$ gives
\begin{equation}\label{HALDIASANGEETAeq10g}
\begin{split}
 \|g - C \|_{\infty} \le ~& \frac{|\alpha|_\infty}{s(1-|\alpha|_\infty)}\Big\{|u|_\infty M+
\frac{1}{4}\big[(3|u|_\infty+|v|_\infty)M+|u|_\infty\times\\~&(h
|d|_\infty+|I|\max\{|d_1|,|d_N|\})\big]\Big\}.
\end{split}
\end{equation}
The desired error estimate is obtained from Theorem 7.1 of \cite{SHH} and (\ref{HALDIASANGEETAeq10g}).
\end{proof}
\noindent \textbf{Convergence result:}
Due to the principle of construction of a smooth FIF, for $g \in
\mathcal{C}^1(I)$, we impose $ |\alpha_n| <  a_n =
\frac{ h_n}{x_N - x_1}$. Hence, $ |\alpha|_{\infty} <
\frac{ h} {x_N - x_1}$, and consequently $g$ converges
uniformly to the original function when the norm of the partition
tends to zero. If we take $ |\alpha_n| <  a_n^k$, then  $\|g - C \|_{\infty}=O(h^k)$ as $h\rightarrow 0$ for $k=2,3$.
\section{Parameter Identification for Constrained Interpolation}\label{HALDIASANGEETAsec4}
In this section, we take up the problem of identifying the parameters of the rational FIF so that the corresponding $\mathcal{C}^1$-RCSFIF enjoys certain desirable shape properties. We  identify suitable
values for the parameters of the rational IFS so that the
corresponding $\mathcal{C}^1$-RCSFIF preserves monotonicity and convexity in Section \ref{HALDIASANGEETAsubsec4d} and Section \ref{HALDIASANGEETAsubsec4e}, respectively.
\subsection{Monotonicity Preserving RCSFIF}\label{HALDIASANGEETAsubsec4d}
\noindent We consider a  data set
$\{(x_i, y_i, d_i):i\in \mathbb{N}_N\}$ such that $y_1\leq
y_2\leq\dots\leq y_N $ (i.e., $\Delta_n\geq 0~ \forall~ n \in \mathbb{N}_{N-1}$).  We derive sufficient conditions on the parameters of the rational IFS so that the corresponding RCSFIF developed in Section \ref{HALDIASANGEETAsec3} generate monotonic fractal curves for a given set of monotonic data.  For a monotonic increasing interpolant $g\in
\mathcal{C}^1(I)$, it is necessary to have $d_i \ge 0 , i\in \mathbb{N}_N$.
We know that a differentiable function $g$ is monotonic increasing on $I$ if and only if $g^{(1)}(x) \ge 0 $ for all $ x \in I$. Calculation of $g^{(1)}\big(L_n(x)\big)$ from (\ref{HALDIASANGEETAeq8}) and further simplifications give:
\begin{equation}\label{HALDIASANGEETAeq25}
g^{(1)}\big(L_n(x)\big) = \frac{\alpha_n}{a_n} g^{(1)}(x) + \frac{{\underset{j=0}{\overset{4} \sum}A_{jn}\theta^j(1-\theta)^{4-j}}}{[u_n +v_n \theta(1-\theta)]^2}, ~~x\in I,~n\in \mathbb{N}_{N-1},
\end{equation}
\begin{equation*}
\begin{split}
A_{0n}=~&u_n^2 [d_n -\frac{\alpha_n}{h_n} (x_N - x_1) d_1],\\
A_{1n}= ~&(6u_n^2+2u_nv_n)[\Delta_n-\frac{\alpha_n}{h_n} (y_N - y_1)]-2u_n^2 [d_{n+1} - \frac{\alpha_n}{h_n} (x_N - x_1) d_N],\\
A_{2n}=~& (12u_n^2+6u_nv_n+v_n^2)[\Delta_n-\frac{\alpha_n}{h_n} (y_N - y_1)]-(3 u_n^2+u_nv_n)[d_n -\frac{\alpha_n}{h_n}\times\\
~&(x_N - x_1) d_1]-(3 u_n^2+u_nv_n)[d_{n+1} - \frac{\alpha_n}{h_n} (x_N - x_1) d_N],\\
A_{3n}=~& (6u_n^2+2u_nv_n)[\Delta_n-\frac{\alpha_n}{h_n} (y_N - y_1)]-2u_n^2[d_n -\frac{\alpha_n}{h_n} (x_N - x_1) d_1],\\
A_{4n}=~& u_n^2[d_{n+1} - \frac{\alpha_n}{h_n} (x_N - x_1) d_N].
\end{split}
\end{equation*}
To maintain  positivity of $g^{(1)}$  in the successive iterations and to keep the desired data dependent monotonicity condition to be simple enough, we assume $\alpha_n \ge 0$ for all $n \in \mathbb{N}_{N-1}$. It follows that for $g^{(1)} \ge 0,$ it is enough
to prove $g^{(1)}\big(L_n(x)\big) \ge 0$ for all $n \in \mathbb{N}_{N-1}$
and $ x \in I$, whenever $g^{(1)}(x) \ge 0$. Then, for $n \in \mathbb{N}_{N-1}$ and an arbitrary knot point $x_j$, sufficient
conditions for $g^{(1)}\big(L_n(x_j)\big)\ge 0$ are
\begin{equation}\label{HALDIASANGEETAeq26}
A_{0n} \ge 0 , A_{1n} \ge 0, A_{2n} \ge 0 , A_{3n} \ge 0, A_{4n} \ge 0,
\end{equation}
where the necessary condition on the derivative parameters are assumed.\\
It is plain to see that the additional conditions on the scaling factors $\alpha_n$ and shape parameters $u_n>0$ and $v_n>0$ prescribed in the following theorem ensure the positivity of  $A_{0n}$, $A_{1n}$, $A_{2n}$, $A_{3n}$ and  $A_{4n}$.
\begin{theorem}\label{RCFIFthm4}
Let $g$  be the RCSFIF defined as in (\ref{HALDIASANGEETAeq8}) associated with a given set of monotonic data $\{(x_i,y_i,d_i):i\in \mathbb{N}_N\}$,  and  let $d_i$, $i \in \mathbb{N}_N$, be chosen so as to satisfy the necessary monotonicity condition. Then the following conditions on the scaling factors and the shape parameters $u_n >0, v_n >0$  on each subinterval $I_n$ are sufficient for $g$ to be monotone on $I$:
\begin{equation}\label{HALDIASANGEETAeq33}
 0 \le \alpha_n \le \min \Big \{a_n, \frac {h_nd_n} { d_1 (x_N - x_1)},    \frac {h_nd_{n+1}} { d_N (x_N - x_1)},   \frac{h_n \Delta_n}{y_N - y_1}\Big
 \},
\end{equation}
\begin{equation}\label{HALDIASANGEETAeq34}
\begin{split}
v_n &\ge \max \Big\{\dfrac{u_n[d_n - \frac{\alpha_n}{h_n} (x_N - x_1) d_1]}{\Delta_n-\frac{\alpha_n}{h_n} (y_N - y_1)},\dfrac{u_n[d_{n+1} - \frac{\alpha_n}{h_n} (x_N - x_1) d_N]}{\Delta_n-\frac{\alpha_n}{h_n} (y_N - y_1)},\\&~~~~~~~~~~~~~~~~~\dfrac{u_n[d_n+d_{n+1} - \frac{\alpha_n}{h_n} (x_N - x_1) (d_1+d_N)]}{\Delta_n-\frac{\alpha_n}{h_n} (y_N - y_1)} \Big \}, n \in \mathbb{N}_{N-1}.
\end{split}
\end{equation}
\end{theorem}
\begin{remark}\label{HALDIASANGEETArem6}
If $\Delta_n = 0$, then we take $\alpha_n =0$ for the monotonicity
 of the FIF $g$. Also in this case, $d_n= d_{n+1}=0$. Consequently, $g\big(L_n(x)\big) = y_n
 = y_{n+1}$ ,i.e., to say that  $g$ reduces to a constant on the
 interval $I_n=[x_n, x_{n+1}]$.
\end{remark}
\begin{remark}\label{HALDIASANGEETArem7}
When all $\alpha_n = 0$, the RCSFIF $g$ reduces to the
classical rational cubic spline $C$. In this case,  condition
(\ref{HALDIASANGEETAeq33}) is obviously true, and the condition (\ref{HALDIASANGEETAeq34})
reduces to
\begin{equation}\label{HALDIASANGEETAeq34}
\begin{split}
v_n &\ge \max \Big\{\dfrac{u_nd_n}{\Delta_n},\dfrac{u_nd_{n+1} }{\Delta_n},\dfrac{u_n(d_n+d_{n+1})}{\Delta_n} \Big \}, n \in \mathbb{N}_{N-1}.
\end{split}
\end{equation}
Thus (\ref{HALDIASANGEETAeq34}) provides sufficient condition for the monotonicity
of  $C$ (\cite{SHH}, p. 78).
\end{remark}
\subsection{Convexity Preserving RCSFIF}\label{HALDIASANGEETAsubsec4e}
Let $\{(x_i,y_i,d_i):i\in \mathbb{N}_N\}$ be the convex data defined over the interval $I$ such that
\begin{equation}\label{HALDIASANGEETAeq35}
 d_1 < \Delta_1 < d_2 < \Delta_2 < \dots < d_i < \Delta_i < \dots <
 d_N.
\end{equation}
We restrict the scaling factors to be nonnegative due to the computational complexity involved. By the principle of
construction of twice differentiable FIFs, we take
$|\alpha_n|<a_n^2$ for all $n\in \mathbb{N}_{N-1}$. Informally,
\begin{equation}\label{HALDIASANGEETAeq36}
g^{(2)}\big(L_n(x)\big) = \frac{\alpha_n}{a_n^2} g^{(2)}(x) + R_n(x),  ~~x\in I,
\end{equation}
where
\begin{equation*}
\begin{split}
R_n(x)=~&\frac{{\underset{j=0}{\overset{5} \sum}B_{jn}\theta^j(1-\theta)^{5-j}}}{h_n[u_n +v_n \theta(1-\theta)]^3},\\
B_{0n} =~&2u_n^2 \big[(3u_n+v_n)\{\Delta_n-\frac{\alpha_n}{h_n} (y_N - y_1)\}-u_n
\{d_{n+1} - \frac{\alpha_n}{h_n} (x_N - x_1) d_N\}\\~&-(2u_n+v_n)\{d_n -\frac{\alpha_n}{h_n} (x_N - x_1) d_1\}\big],\\
B_{1n}=~& 2u_n^2\big[7u_n\big\{\Delta_n-d_n-\frac{\alpha_n}{h_n}\{ (y_N - y_1)-(x_N - x_1) d_1\}\big\}+2v_n\big\{\Delta_n-d_n\\~&-\frac{\alpha_n}{h_n}\{ (y_N - y_1) - (x_N - x_1) d_1\}\big\}+2u_n\big\{\Delta_n-d_{n+1}+\frac{\alpha_n}{h_n}\{(x_N - x_1) d_N \\~&-(y_N - y_1)\}\big\}\big],\\
B_{2n}=~& 2u_n\big[(6u_n^2+u_nv_n)\big\{\Delta_n-\frac{\alpha_n}{h_n} (y_N - y_1) \big\}-(8u_n^2+u_nv_n)\{d_n -\frac{\alpha_n}{h_n}\times\\~& (x_N - x_1) d_1\}+ 2u_n^2 \{d_{n+1} - \frac{\alpha_n}{h_n} (x_N - x_1) d_N\}\big],\\
B_{3n}=~& 2u_n\big[(-6u_n^2-u_nv_n)\big\{\Delta_n-\frac{\alpha_n}{h_n} (y_N - y_1) \big\}+(8u_n^2+u_nv_n)\{d_{n+1}\\~&-\frac{\alpha_n}{h_n} (x_N - x_1) d_N\}- 2u_n^2 \{d_n - \frac{\alpha_n}{h_n} (x_N - x_1) d_1\}\big],\\
B_{4n}=~& 2u_n^2\big[7u_n\big\{d_{n+1}-\Delta_n-\frac{\alpha_n}{h_n}\{(x_N - x_1) d_N -(y_N - y_1) \}\big\}+2v_n\big\{d_{n+1}-\\~&\Delta_n-\frac{\alpha_n}{h_n}\{(x_N - x_1) d_N -(y_N - y_1) \}\big\}-2u_n\big\{\Delta_n-d_n-\frac{\alpha_n}{h_n}\{ (y_N - y_1)\\~& - (x_N - x_1) d_1\}\big\}\big],\\
B_{5n}=~&2u_n^2 \big[-(3u_n+v_n)\{\Delta_n-\frac{\alpha_n}{h_n} (y_N - y_1)\}+u_n
\{d_n - \frac{\alpha_n}{h_n} (x_N - x_1) d_1\}\\~&+(2u_n+v_n)\{d_{n+1} -\frac{\alpha_n}{h_n} (x_N - x_1) d_N\}\big].
\end{split}
\end{equation*}
\normalsize Recall that for $n\in \mathbb{N}_{N-1}$, the maps $L_n:[x_1,x_N]
\rightarrow [x_n,x_{n+1}]$ satisfy $L_n(x_1)=x_n$ and
$L_n(x_N)=x_{n+1}$. Therefore, we obtain
\begin{eqnarray}\label{HALDIASANGEETAeq37}
g^{(2)}(x_1^+)&=& [1-\frac{\alpha_1}{a_1^2}]^{-1}\frac{B_{01}}{h_1u_1^3},~~ g^{(2)}(x_N^-)= [1-\frac{\alpha_{N-1}}{a_{N-1}^2}]^{-1}\frac{B_{5,
N-1}}{h_{N-1}u_{N-1}^3
}, \nonumber\\
g^{(2)} (x_j^+)
&=& \frac{\alpha_j}{a_j^2}g^{(2)}(x_1^+)+\frac{B_{0j}}{h_ju_j^3},\;\; j=2,3,\dots, N-1.
\end{eqnarray}
For $0 \le \alpha_n <a_n^2$, it follows from (\ref{HALDIASANGEETAeq37})
that if $B_{0n} \ge 0$~ ($n\in \mathbb{N}_{N-1}$) and $B_{5, N-1} \ge 0$, then
the second derivatives (right-handed) at the knots $x_n$, $ n\in \mathbb{N}_{N-1}$, and the second derivative (left-handed) at $x_N$ are
nonnegative. For a knot point $x_j$, $j \in \mathbb{N}_{N-1},$ we have
\begin{eqnarray*}\label{HALDIASANGEETAeq38}
g^{(2)}\big(L_n(x_j)^+\big)= \frac{\alpha_n}{a_n^2} g^{(2)}(x_j^+)
+ R_n(x_j).
\end{eqnarray*}
Whence, with the assumption $B_{0n} \ge 0$ for all $n\in \mathbb{N}_{N-1}$, we
have $g^{(2)}\big(L_n(x_j)^+\big) \ge 0$, provided $R_n(x_j)\ge
0$. Note that $R_n(x_j)\ge 0$ is satisfied if the coefficients
$B_{mn} \ge 0$ for $m=0,1,\dots, 5$.
\begin{theorem}\label{RCFIFthm5}
Suppose $\{(x_i,y_i,d_i):i \in \mathbb{N}_N\}$ is a set of strictly convex
data, and $g$ is the corresponding rational cubic spline FIF
described in (\ref{HALDIASANGEETAeq8}). Assume that the derivative
parameters at the knots satisfy
$d_1<\Delta_1<\dots<d_n<\Delta_n<d_{n+1}<\dots<\Delta_{N-1}<d_N$.
Then, the following conditions on the scaling factors and  the
shape parameters  are sufficient for the convexity of $g$ on  $ I=[x_1, x_N].$
\begin{equation*}
\begin{split}
0\leq\alpha_n &<\min\left\{a_n^2,\frac{h_n(\Delta_n-d_n)}{y_N-y_1-d_1(x_N-x_1)},
\frac{h_n(d_{n+1}-\Delta_n)}{d_N(x_N-x_1)-(y_N-y_1)}\right\},\\
v_n \geq &\max \Big\{
u_n\dfrac{d_{n+1}-\frac{\alpha_n}{h_n}d_N(x_N-x_1)-\big[\Delta_n-\frac{\alpha_n}{h_n}
(y_N-y_1)\big]}{\Delta_n-
\frac{\alpha_n}{h_n}(y_N-y_1)-\big[d_n-\frac{\alpha_n}{h_n}d_1(x_N-x_1)\big]},\\
&~~~~~~~~~~~~~u_n\dfrac{\Delta_n-\frac{\alpha_n}{h_n}(y_N-y_1)-
\big[d_n-\frac{\alpha_n}{h_n}
d_1(x_N-x_1)\big]}{d_{n+1}-\frac{\alpha_n}{h_n}
d_N(x_N-x_1)-\big[\Delta_n-\frac{\alpha_n}{h_n}(y_N-y_1)\big]}\Big\}~ \forall~n \in \mathbb{N}_{N-1}.
\end{split}
\end{equation*}
\end{theorem}
\begin{remark}\label{HALDIASANGEETArem8}
If the given set of data is not strictly convex but $\Delta_n- d_n
= 0$ or $d_{n+1}-\Delta_n =0$, then we take $\alpha_n = 0$. Now
for $g^{(2)}(x) \geq 0$ (see the expressions for the coefficients
$B_{mn}$), we take $d_n = d_{n+1}= \Delta_n$. Thus, we get
$g\big(L_n(x)\big)= \frac{(x_N-x)y_n + (x-x_1) y_{n+1}}{x_N-x_1}$,
i.e., the interpolant reduces to a straight line segment on the
interval $[x_n, x_{n+1}]$.
\end{remark}
\begin{remark}\label{HALDIASANGEETArem9}
When $\alpha_n=0$ for all $n \in \mathbb{N}_{N-1}$, Theorem \ref{RCFIFthm5}
recaptures the sufficient conditions for  the convexity of the
classical rational cubic spline $C$  described in
\cite{SHH}.
\end{remark}
\section{Numerical Examples}\label{HALDIASANGEETAsec6}
RCSFIF lying within the rectangle $[0,1]\times[0.1,5]$. Our choices of the scaling factors and shape parameter values are displayed in Table \ref{chapter3table1}, and the corresponding range restricted RCSFIFs are generated in Figs. \ref{chapter3fig1}(d)-(e).\\
Remark \ref{HALDIASANGEETArem5}.
To illustrate the monotonicity preserving RCSFIF scheme appeared in Section \ref{HALDIASANGEETAsubsec4d}, we take a monotonic data set $\{(x_i,y_i)\}_{i=1}^{4}$=$\{(0,124), (0.5,331), (2.2,379), (3.3,835\}$, reported in \cite{SHH}. The derivative values $d_1=501.6738,d_2=326.3262,  d_3=262.7807,\\  d_4=566.3102$ are estimated using the amm. For monotonic FIFs, the computed bounds on the scaling factors are $0\le \alpha_1 < 0.0873$, $0\le \alpha_2 < 0.067$, $0\le \alpha_1 < 0.1746$ as prescribed in Theorem \ref{RCFIFthm4}. We take monotonic RCSFIF $g_{11}$ in Fig. \ref{chapter3fig2}(a) as our reference curve generated by iterating the IFS code with parameters displayed in Table \ref{chapter3table2}. We compare the effect of changing the value of parameters in a specified interval. Changing $\alpha_1$ to $0.01$ (see Table \ref{chapter3table2}),
we obtain RCSFIF $g_{12}$ in Fig. \ref{chapter3fig2}(b). It is clearly visible that the perturbation in $\alpha_1$ effects the RCSFIFs considerably
in the interval $[x_1, x_2]$, whereas there is  no noticeable
change in other subintervals. It can be observed that
changes in $\alpha_2$, $\alpha_3$ and $v_1$ produce  local effects when we compare RCSFIF $g_{13}$, $g_{14}$ and $g_{15}$ appeared in Fig. \ref{chapter3fig2}(c), Fig. \ref{chapter3fig2}(d) and Fig. \ref{chapter3fig2}(e), respectively with RCSFIF $g_{11}$ in Fig. \ref{chapter3fig2}(a). By taking
$\alpha_n=0$ for all $n \in \mathbb{N}_{N-1}$, we recover a standard monotonic
rational cubic spline plotted in Fig. \ref{chapter3fig2}(f). The derivative functions  $g^{(1)}_{1n}$, $n=1,2,\dots, 6$ are generated in Figs. \ref{chapter3fig3}(a)-(f). These curves have points of nondifferentiability on finite or dense subset of the interpolation interval $[0,3.3]$. The derivative $g^{(1)}_{16}$ of the classical rational cubic spline is smooth.
\begin{center}
\begin{table}[h!]
\caption{Parameters corresponding to RCSFIFs in Fig. \ref{chapter3fig2}}\label{chapter3table2}
\begin{center}
\scriptsize {
\begin{tabular}{|l |l| l| l|}\hline
$\hspace{0.2cm}Figure \hspace{0.6cm}$&$\hspace{0.2cm}Scaling factors  \hspace{0.6cm}$&$\hspace{0.2cm}Shape parameters\hspace{0.6cm}$\\ \hline
 ~~~\ref{chapter3fig2}(a)&$\alpha=(0.08,0.06,0.15)$&$u=(0.1,0.1,0.1)$,$v=(0.09,15,0.15)$\\\hline
 ~~~\ref{chapter3fig2}(b)&$\alpha=(0.01,0.06,0.15)$&$u=(0.1,0.1,0.1)$,$v=(0.09,15,0.15)$\\ \hline
 ~~~\ref{chapter3fig2}(c)&$\alpha=(0.08,0.01,0.15)$&$u=(0.1,0.1,0.1)$,$v=(0.09,15,0.15)$ \\ \hline
 ~~~\ref{chapter3fig2}(d)&$\alpha=(0.08,0.06,0.01)$&$u=(0.1,0.1,0.1)$,$v=(0.09,15,0.15)$ \\ \hline
 ~~~\ref{chapter3fig2}(e)&$\alpha=(0.08,0.06,0.15)$&$u=(0.1,0.1,0.1)$,$v=(10,15,0.15)$\\\hline
 ~~~\ref{chapter3fig2}(f)&$\alpha=(0, 0,  0)$&$u=(0.1,0.1,0.1)$,$v=(0.09,15,0.15)$\\ \hline
\end{tabular}}
\end{center}
\end{table}
\end{center}
\begin{figure}[h!]
\begin{center}
\begin{minipage}{0.3\textwidth}
\epsfig{file=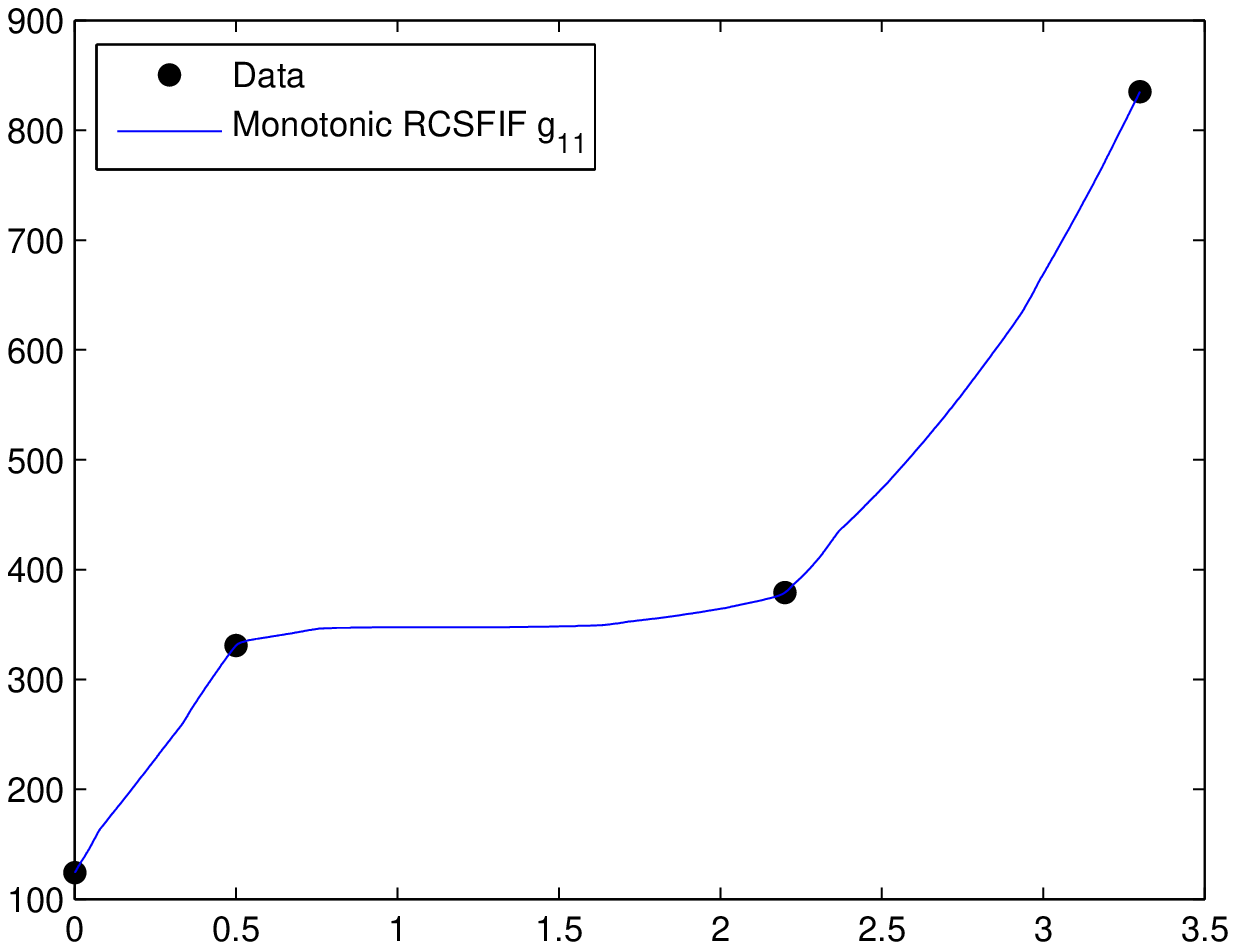,scale=0.25} \centering{\scriptsize{(a)
Monotonic RCSFIF $g_{11}$}}
\end{minipage}\hfill
\begin{minipage}{0.3\textwidth}
\epsfig{file=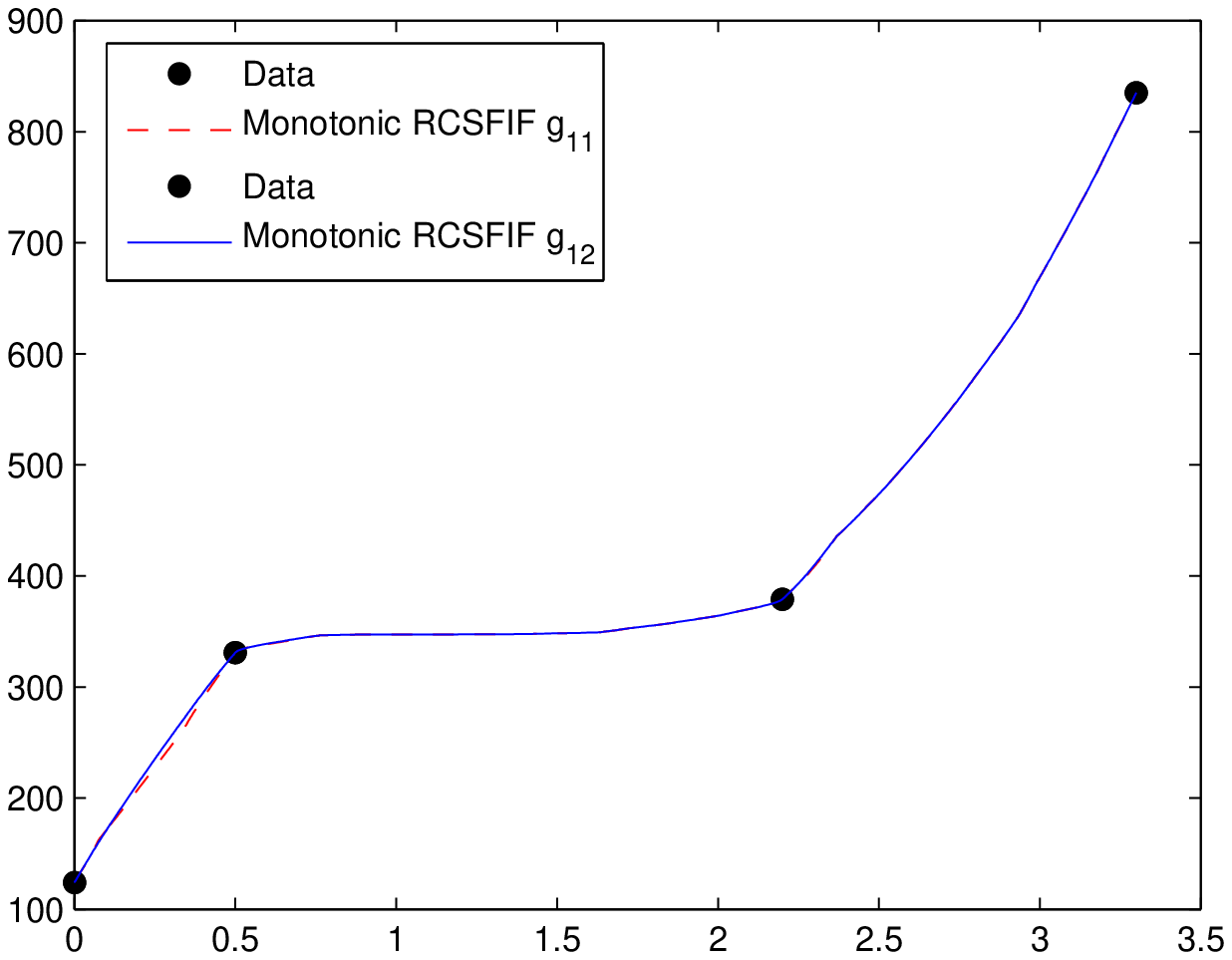,scale=0.25} \centering{\scriptsize{(b)
Monotonic RCSFIF $g_{12}$ (effect of perturbation in $\alpha_1$)}}
\end{minipage}\hfill
\begin{minipage}{0.3\textwidth}
\epsfig{file = 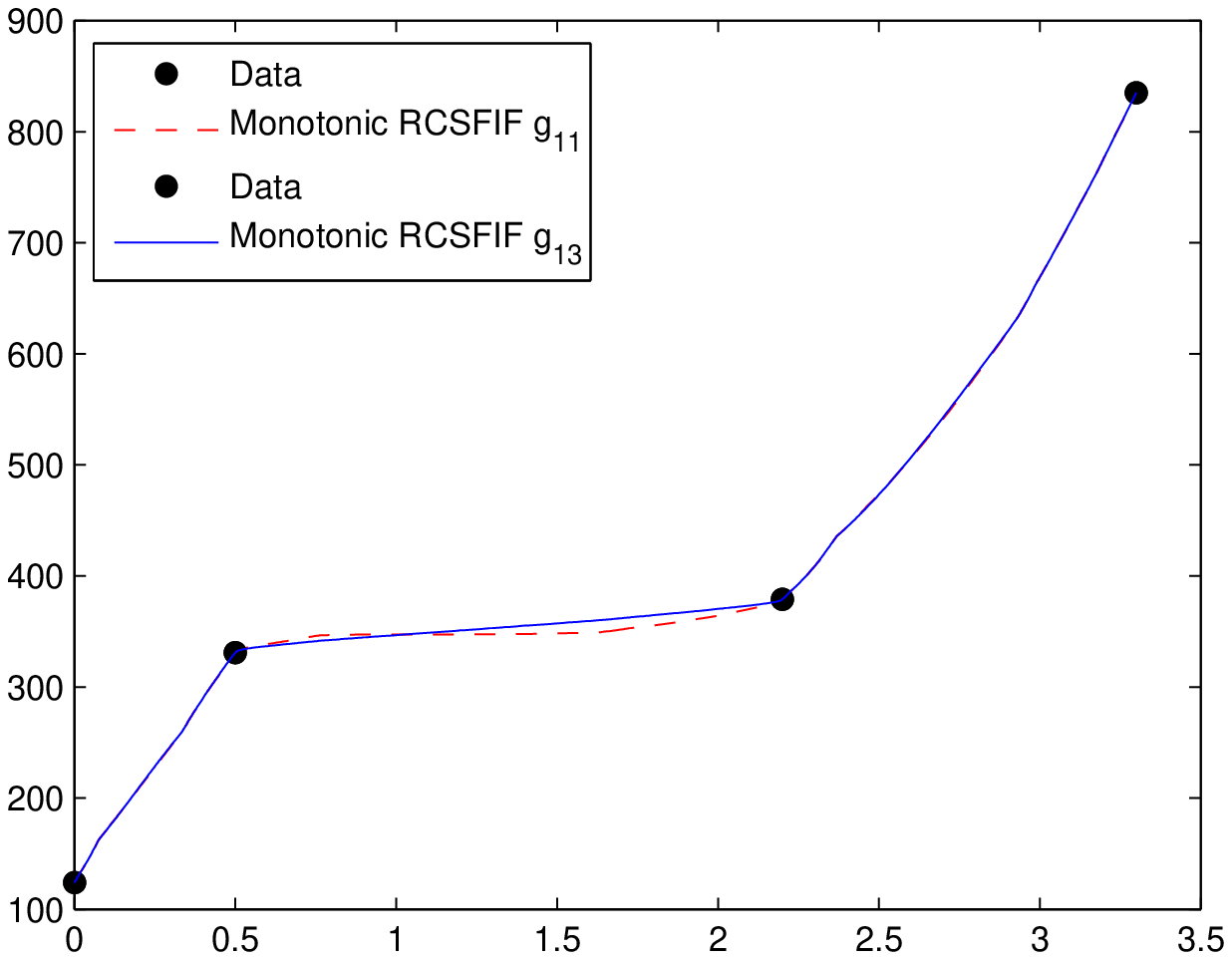,scale=0.25}\\ \centering{\scriptsize{(c)
Monotonic RCSFIF $g_{13}$ (effect of perturbation in $\alpha_2$)}}
\end{minipage}\hfill
\begin{minipage}{0.3\textwidth}
\epsfig{file=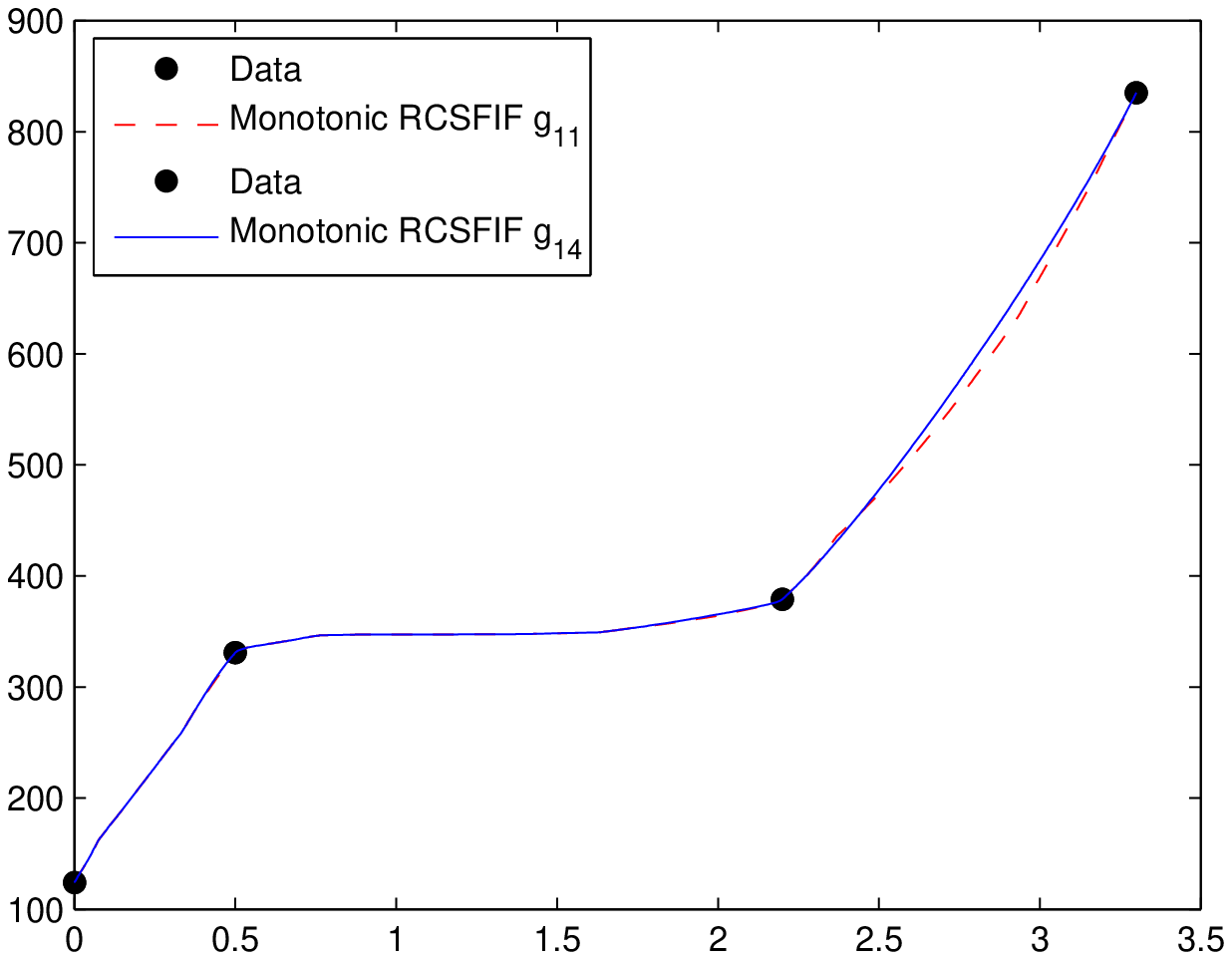,scale=0.25} \centering{\scriptsize{(d)
Monotonic RCSFIF $g_{14}$ (effect of perturbation in $\alpha_3$)}}
\end{minipage}\hfill
\begin{minipage}{0.3\textwidth}
\epsfig{file=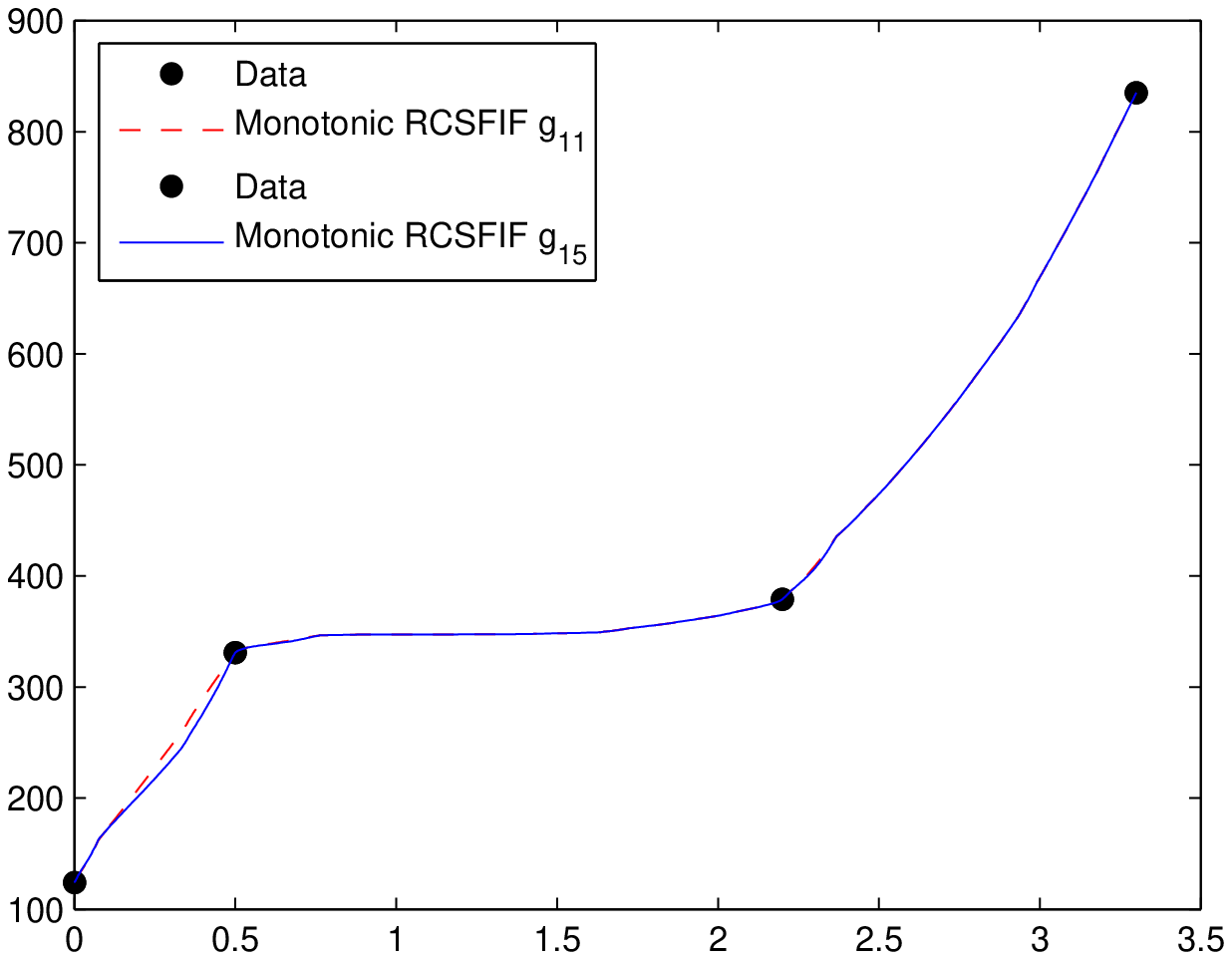,scale=0.25} \centering{\scriptsize{(e)
Monotonic RCSFIF $g_{15}$ (effect of perturbation in $v_1$)}}
\end{minipage}\hfill
\begin{minipage}{0.3\textwidth}
\epsfig{file=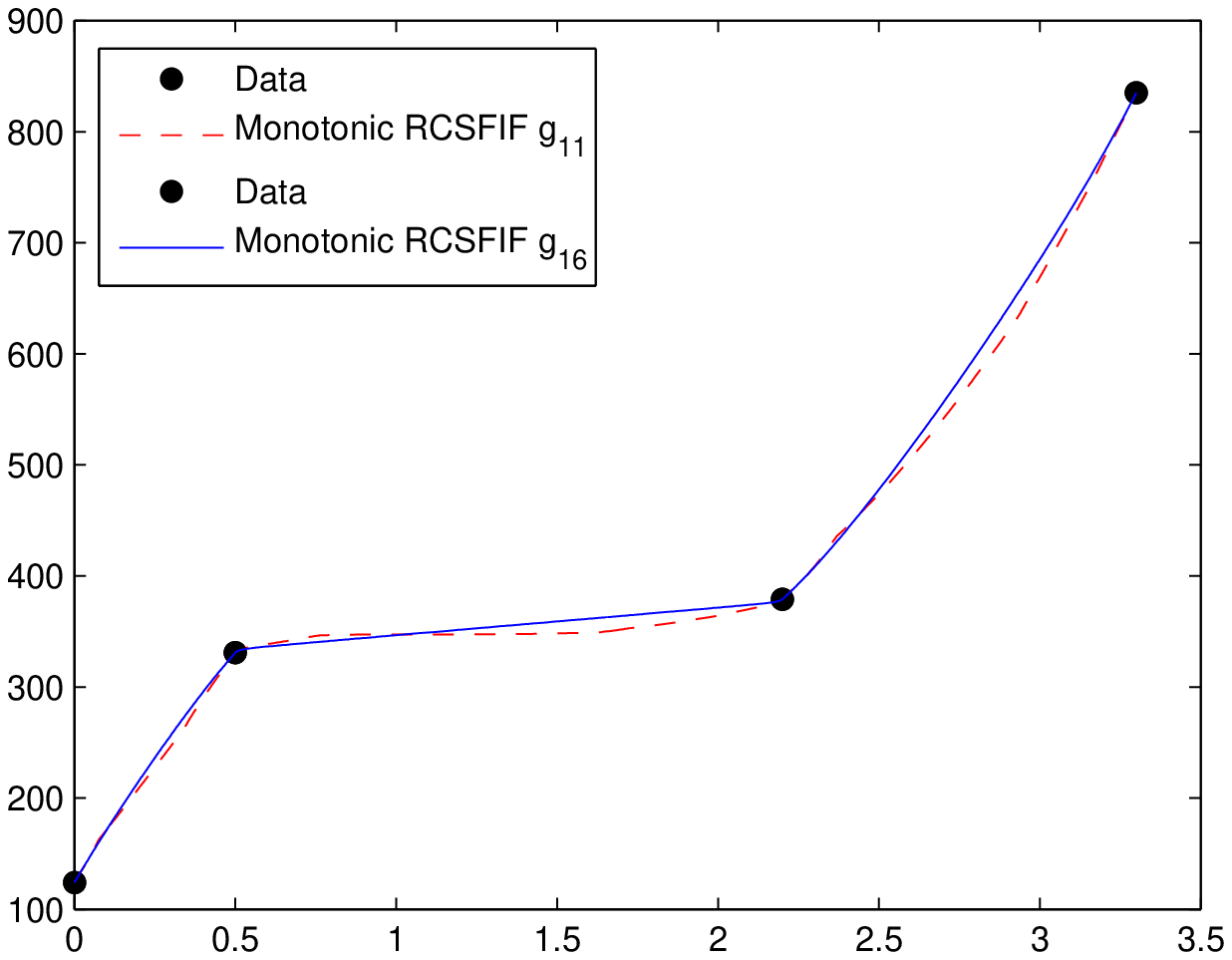,scale=0.25} \centering{\scriptsize{(f)
Classical Monotonic rational cubic spline $g_{16}$}}
\end{minipage}\hfill\\
\caption{Monotonic RCSFIFs $g_{1n}, n=1,\dots,5$ (the
interpolating data points are given by the circles and the relevant RCSFIF by solid lines).}\label{chapter3fig2}
\end{center}
\end{figure}

\begin{figure}[h!]
\begin{center}
\begin{minipage}{0.3\textwidth}
\epsfig{file=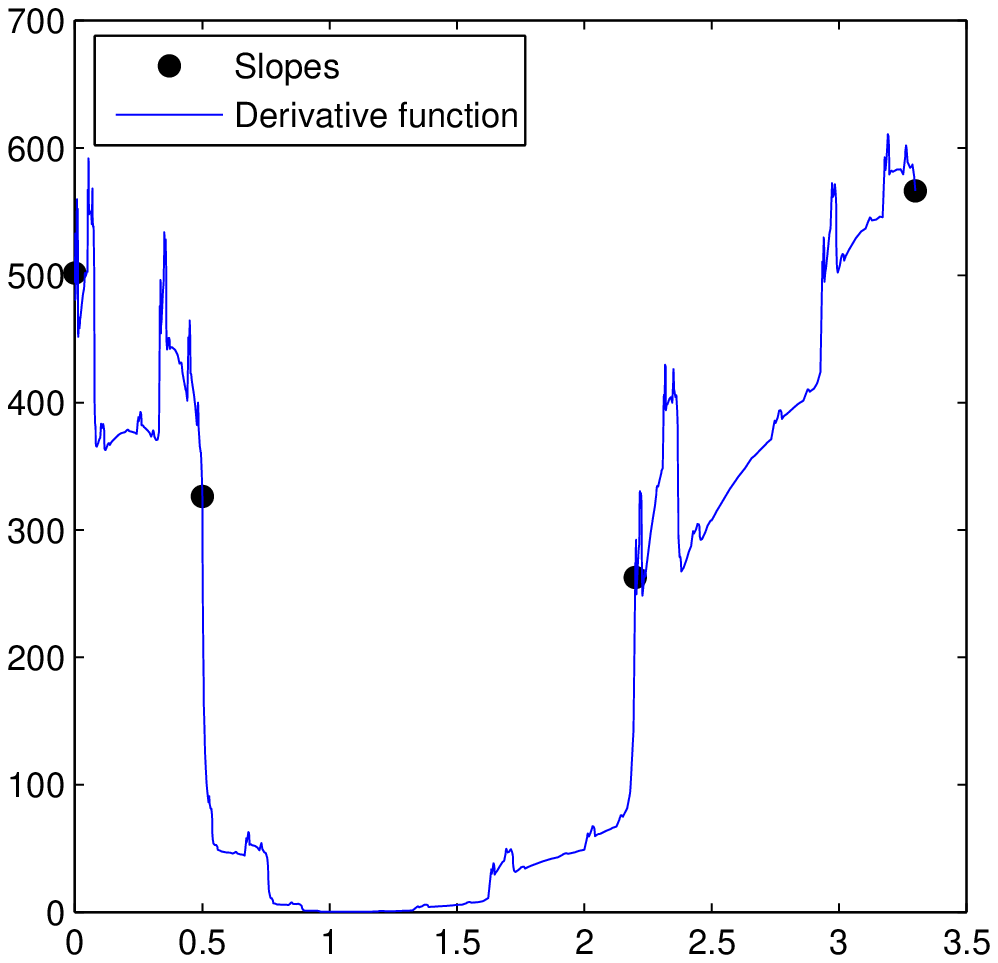,scale=0.25} \centering{\scriptsize{(a)
Fractal function $g^{(1)}_{11}$}}
\end{minipage}\hfill
\begin{minipage}{0.3\textwidth}
\epsfig{file=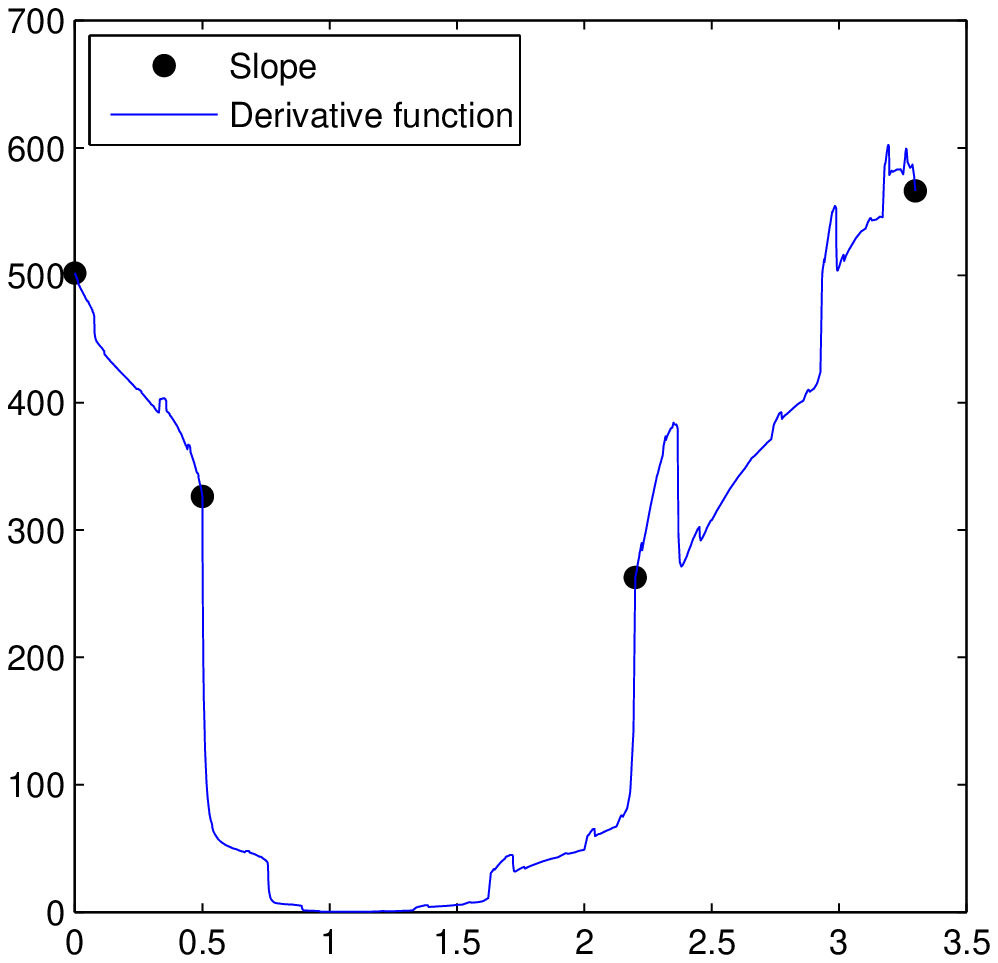,scale=0.25} \centering{\scriptsize{(b)
Fractal function $g^{(1)}_{12}$}}
\end{minipage}\hfill
\begin{minipage}{0.3\textwidth}
\epsfig{file = 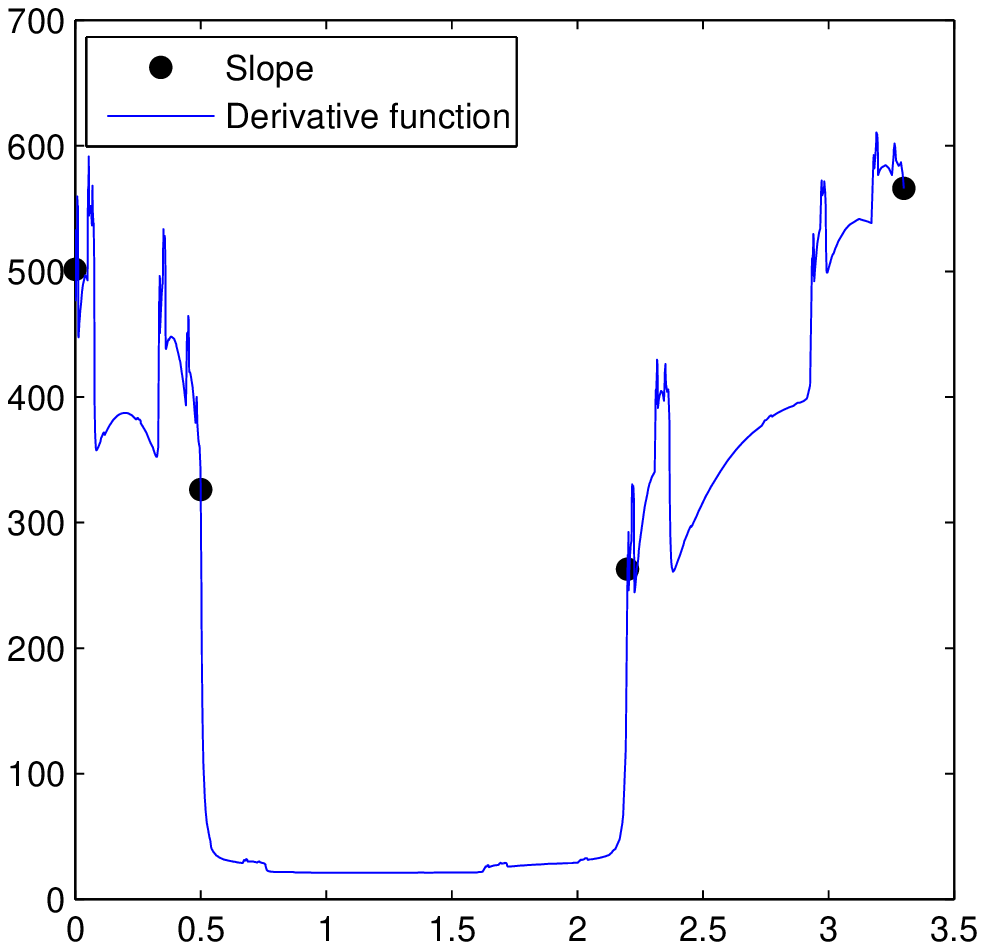,scale=0.25}\\ \centering{\scriptsize{(c)
Fractal function $g^{(1)}_{13}$}}
\end{minipage}\hfill
\begin{minipage}{0.3\textwidth}
\epsfig{file=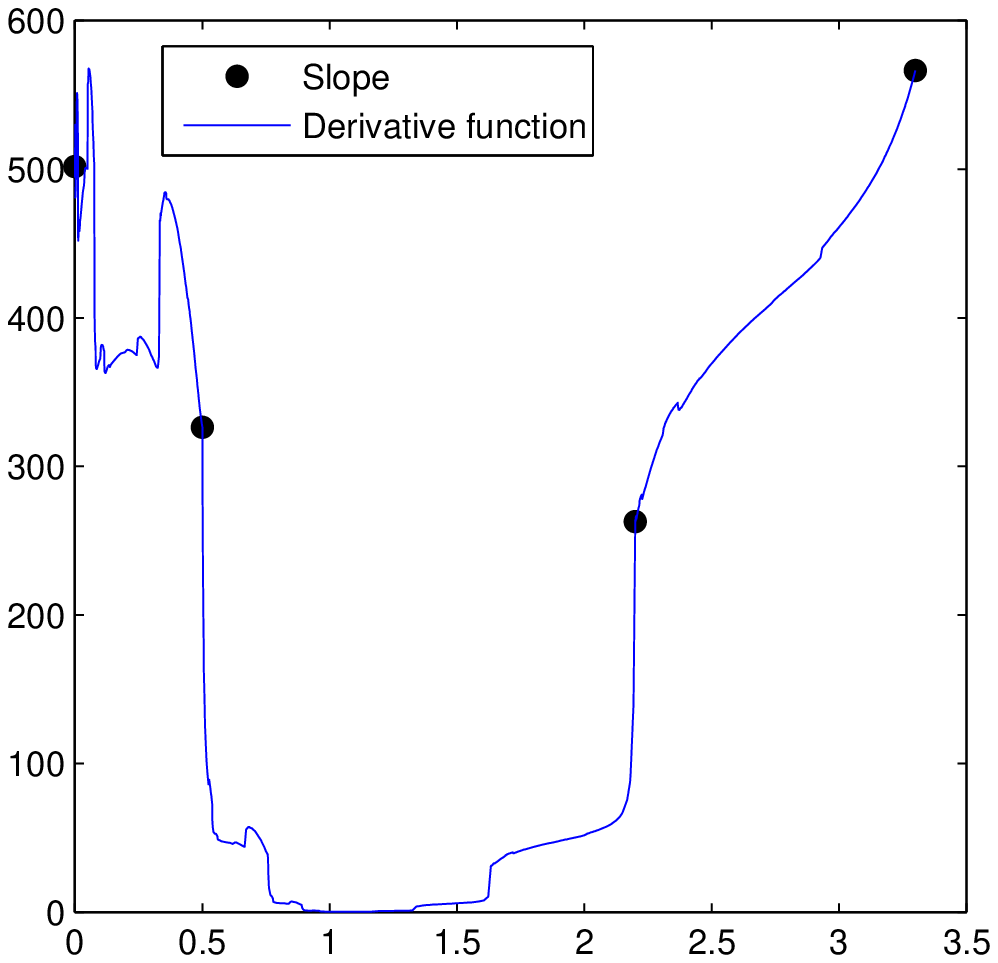,scale=0.25} \centering{\scriptsize{(d)
Fractal function $g^{(1)}_{14}$}}
\end{minipage}\hfill
\begin{minipage}{0.3\textwidth}
\epsfig{file=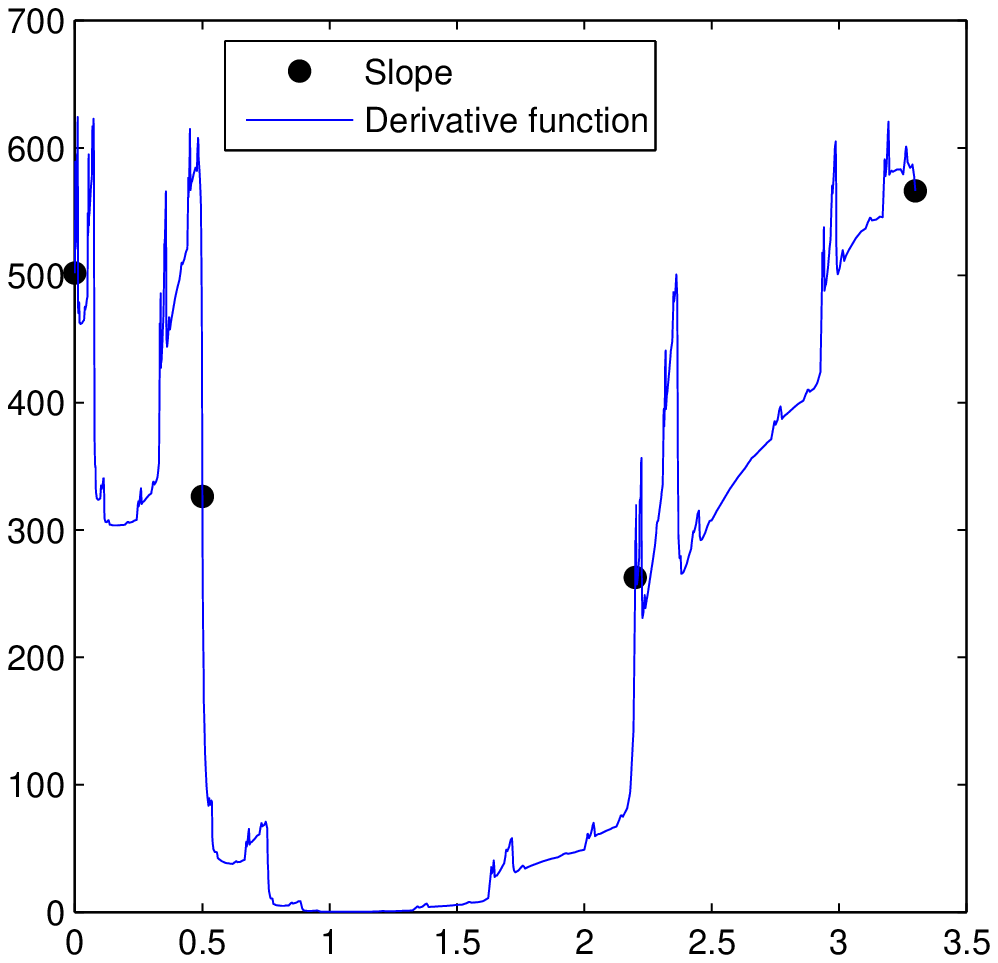,scale=0.25} \centering{\scriptsize{(e)
Fractal function $g^{(1)}_{15}$}}
\end{minipage}\hfill
\begin{minipage}{0.3\textwidth}
\epsfig{file=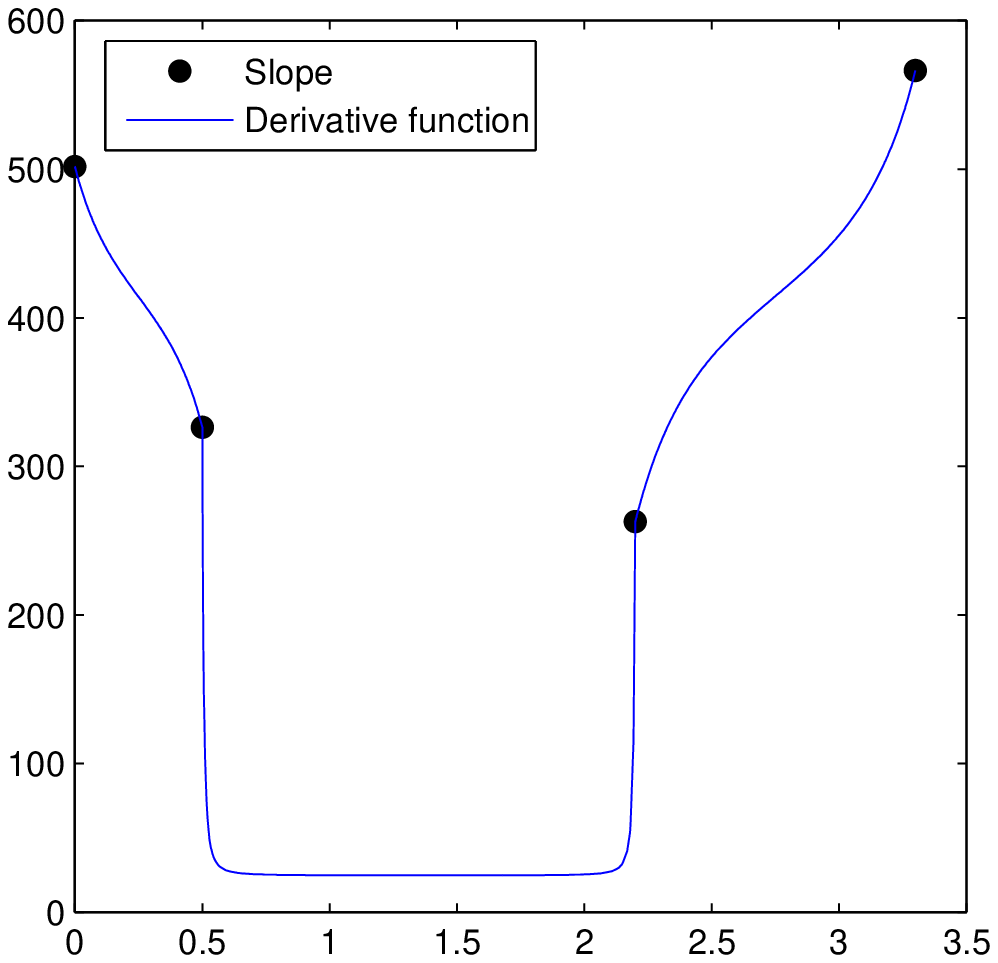,scale=0.25} \centering{\scriptsize{(f)
Function $g^{(1)}_{16}$}}
\end{minipage}\hfill\\
\caption{Derivatives of the monotonic RCSFIFs in Figs.
\ref{chapter3fig2}(a)-(f)}\label{chapter3fig3}
\end{center}
\end{figure}
\par
\noindent Consider a convex data set  $\{(0,0), (0.5,8.7713), (0.75,18.8599), (1,
32.4673)\}$ and the derivative values are estimated using the amm as $d_1=2.3347,d_2=32.7505,  d_3=47.3920,  d_4=61.4672$, which satisfy the necessary convexity conditions:
$d_1<\Delta_1<d_2<\Delta_2<d_3<\Delta_3<d_4$;
$d_1<\frac{y_4-y_1}{x_4-x_1}<d_4$. In Fig. \ref{chapter3fig4}(a), we do not follow the prescription given in Theorem \ref{RCFIFthm5} for which we obtain a non-convex RCSFIF, and the values of scaling factors and shape parameters are shown in Table \ref{chapter3table3}. Since the scaling factors are selected
only to satisfy $|\alpha_n|<a_n^2$, $n=1,2$, and $3$, the
fractal curve in Fig. \ref{chapter3fig4}(a) has undesired inflections in first subinterval. Next we apply Theorem \ref{RCFIFthm5} to get suitable values of
the scaling factors and  the shape parameters that generate convex
RCSFIFs. The computed bounds on the scaling factors are :  $0\leq \alpha_1<0.2500$,
$0\leq\alpha_2<0.0607$, $0\leq\alpha_3<0.0584$. The convex RCSFIF in Fig. \ref{chapter3fig4}(b) is generated with the scaling
factors and shape parameters (see Table \ref{chapter3table3}) according
to Theorem \ref{RCFIFthm5}.  By taking $\alpha_n=0$, $n=1,2,3$ and the shape parameters as in Table \ref{chapter3table3}, a classical rational cubic spline that preserves the convexity of the data is obtained in \ref{chapter3fig4}(c).  To claim that the RCSFIFs relating to other
subintervals are not extremely sensitive towards the changes of
parameters in a particular subinterval, we have taken the same sets of
parameters (see Table \ref{chapter3table3}) except for the scaling factor
in the first subinterval and plotted the \ref{chapter3fig4}(d)-(e). We observe that the curves differ only in the first subinterval. We obtain a convex
RCSFIF in Fig. \ref{chapter3fig4}(f) with negative scalings in all the subintervals so the conditions prescribed by Theorem \ref{RCFIFthm5} are sufficient but not necessary. In general, the 2nd derivative of convex RCSFIFs
are typical fractal functions having points of nondifferentiabilty on finite or dense subset of the interpolation interval.
\begin{center}
\begin{table}[h!]
\caption{Parameters corresponding to RCSFIFs in Fig. \ref{chapter3fig4}}\label{chapter3table3}
\begin{center}
\scriptsize {
\begin{tabular}{|l |l| l| l|}\hline
$\hspace{0.2cm}Figure \hspace{0.6cm}$&$\hspace{0.2cm}Scaling factors  \hspace{0.6cm}$&$\hspace{0.2cm}Shape parameters\hspace{0.6cm}$\\ \hline
 ~~~\ref{chapter3fig2}(a)&$\alpha=(-0.24,0.05,0.04)$&$u=(0.1,0.1,0.1)$,$v=(0.2,0.15,0.14)$\\\hline
 ~~~\ref{chapter3fig2}(b)&$\alpha=(0.24,0.05,0.04)$&$u=(0.1,0.1,0.1)$,$v=(0.2,0.15,0.14)$\\ \hline
 ~~~\ref{chapter3fig2}(c)&$\alpha=(0, 0,  0)$&$u=(0.1,0.1,0.1)$,$v=(0.2,0.15,0.14)$ \\ \hline
 ~~~\ref{chapter3fig2}(d)&$\alpha=(0.24,0.05,0.04)$&$u=(0.2,0.3,0.4)$,$v=(0.3,0.2,0.3)$ \\ \hline
 ~~~\ref{chapter3fig2}(e)&$\alpha=(0.1,0.05,0.04)$&$u=(0.2,0.3,0.4)$,$v=(0.3,0.2,0.3)$\\\hline
 ~~~\ref{chapter3fig2}(f)&$\alpha=(-0.01, -0.01,-0.01  0)$&$u=(2,2,2)$,$v=(3,3,3)$\\ \hline
\end{tabular}}
\end{center}
\end{table}
\end{center}
\begin{figure}[h!]
\begin{center}
\begin{minipage}{0.3\textwidth}
\epsfig{file=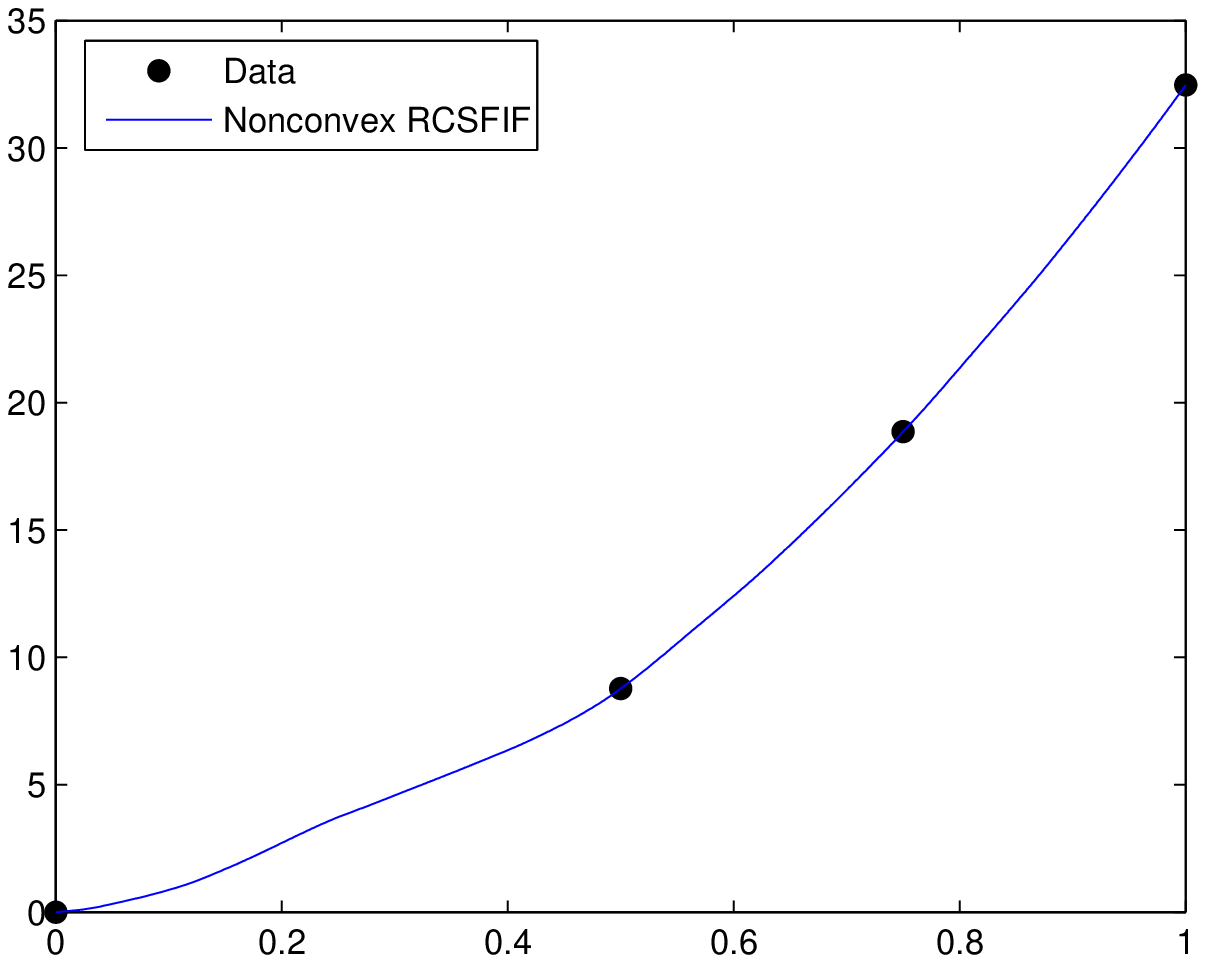,scale=0.25} \centering{\scriptsize{(a)
Non-convex RCSFIF}}
\end{minipage}\hfill
\begin{minipage}{0.3\textwidth}
\epsfig{file=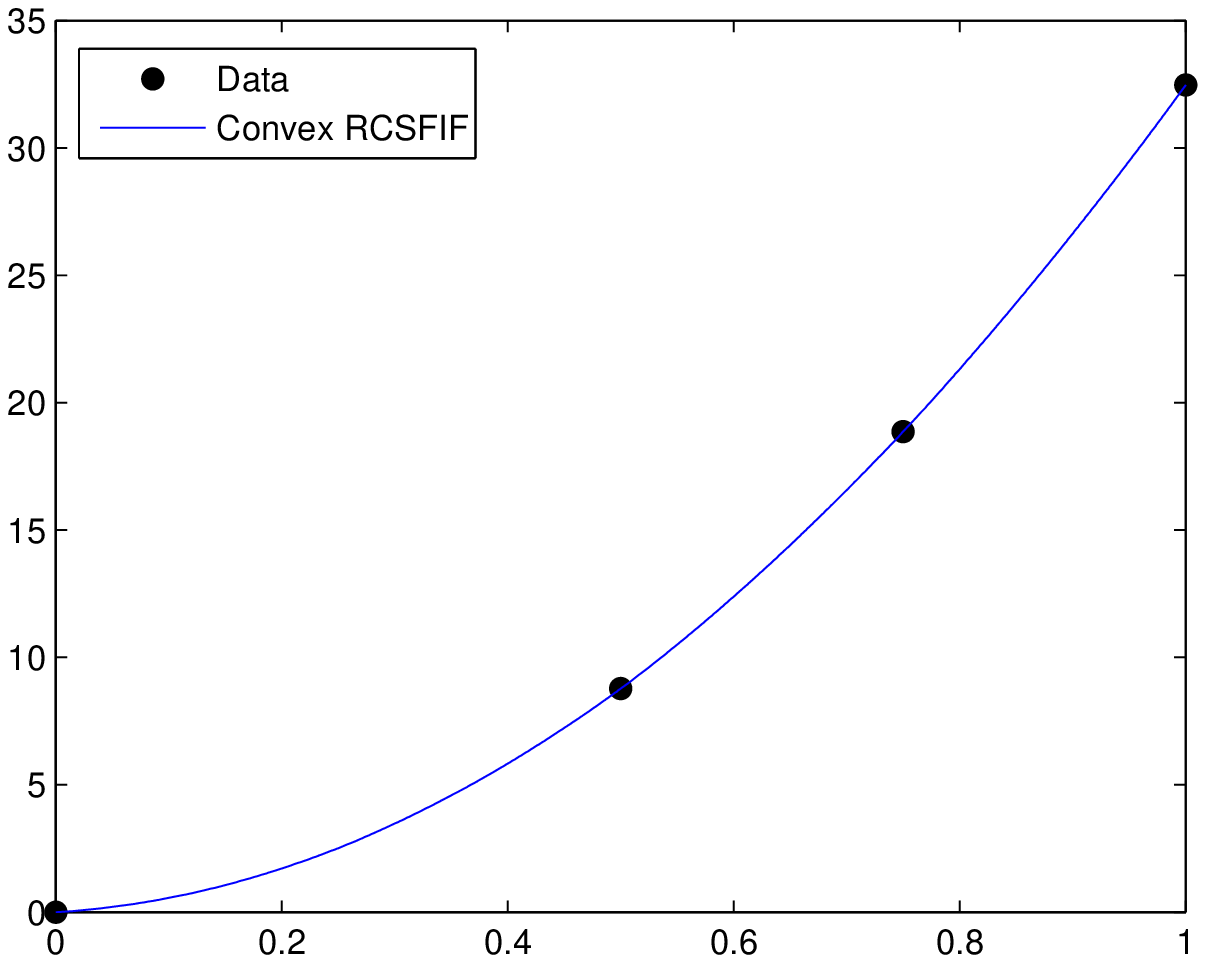,scale=0.25} \centering{\scriptsize{(b)
Convex RCSFIF}}
\end{minipage}\hfill
\begin{minipage}{0.3\textwidth}
\epsfig{file = 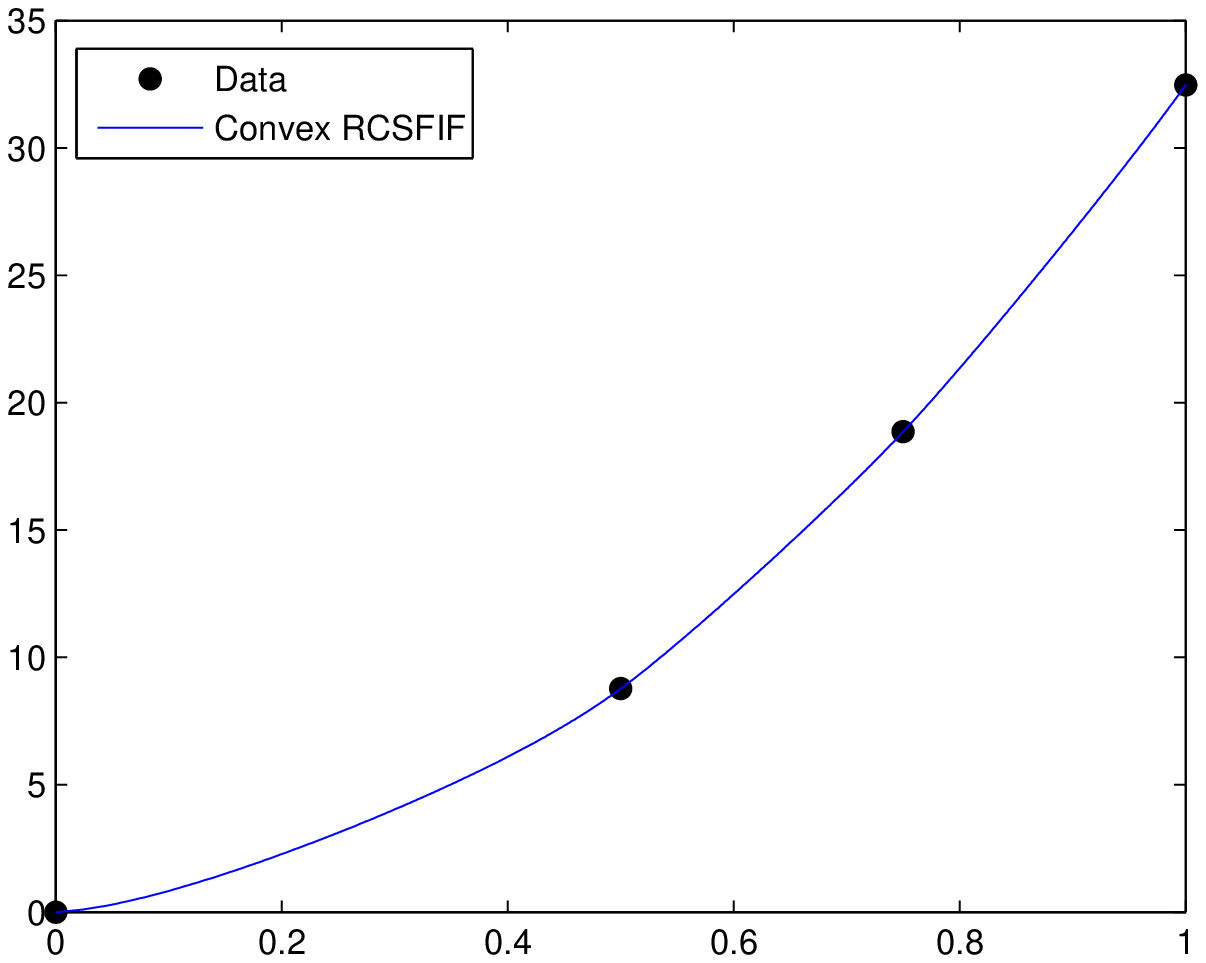,scale=0.25}\\ \centering{\scriptsize{(c)
Classical rational cubic spline}}
\end{minipage}\hfill
\begin{minipage}{0.3\textwidth}
\epsfig{file=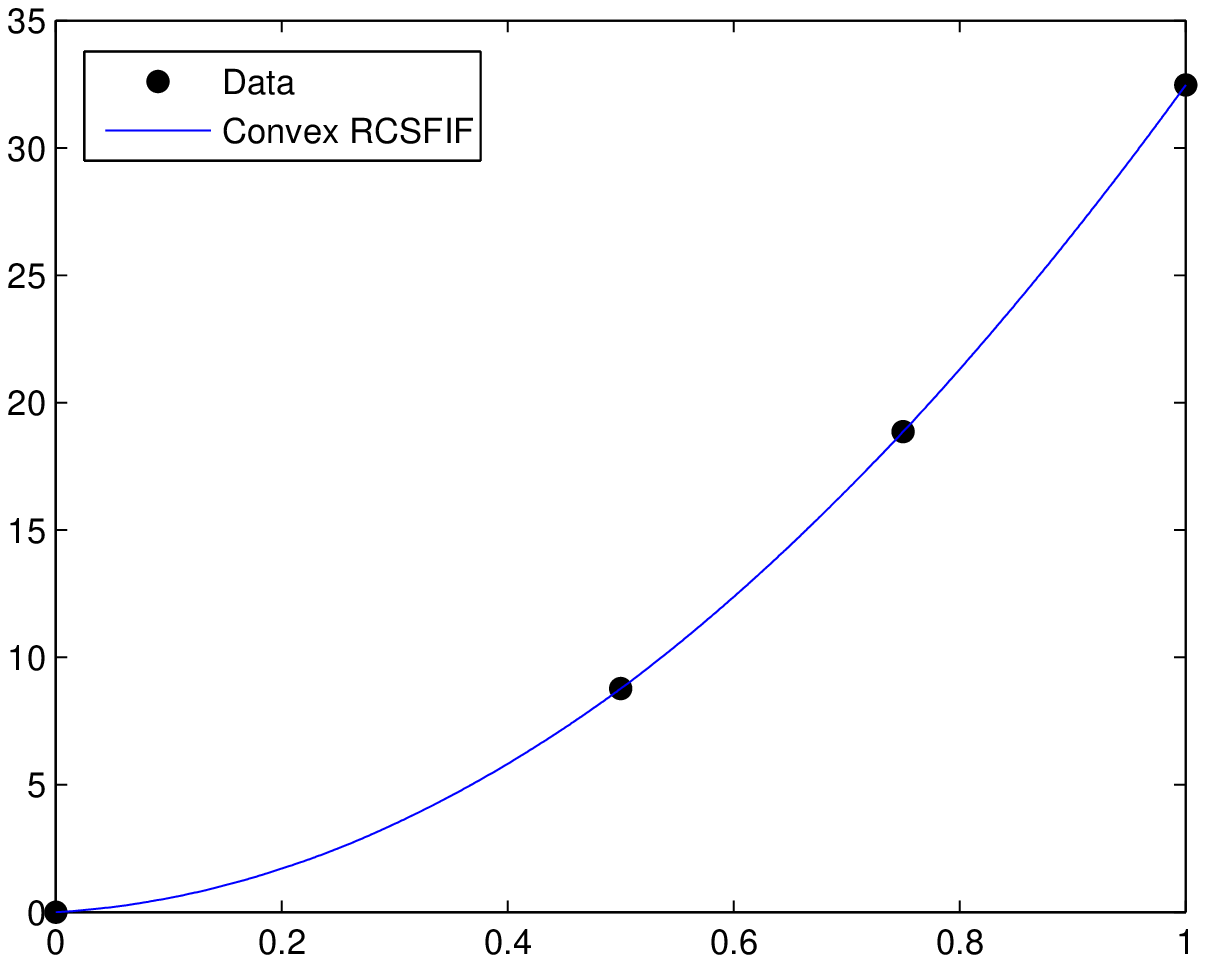,scale=0.25} \centering{\scriptsize{(d)
Convex RCSFIF}}
\end{minipage}\hfill
\begin{minipage}{0.3\textwidth}
\epsfig{file=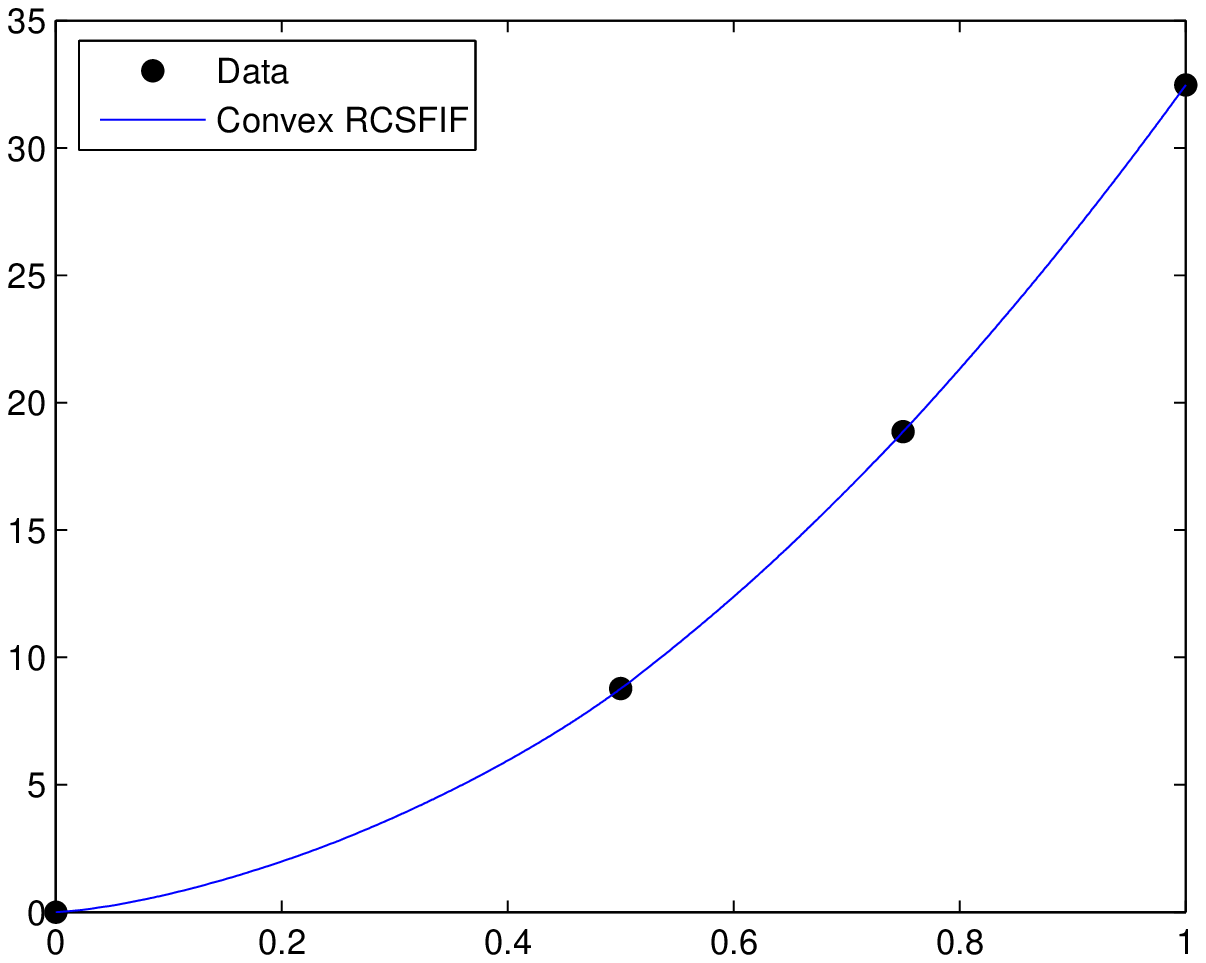,scale=0.25} \centering{\scriptsize{(e)
Convex RCSFIF}}
\end{minipage}\hfill
\begin{minipage}{0.3\textwidth}
\epsfig{file=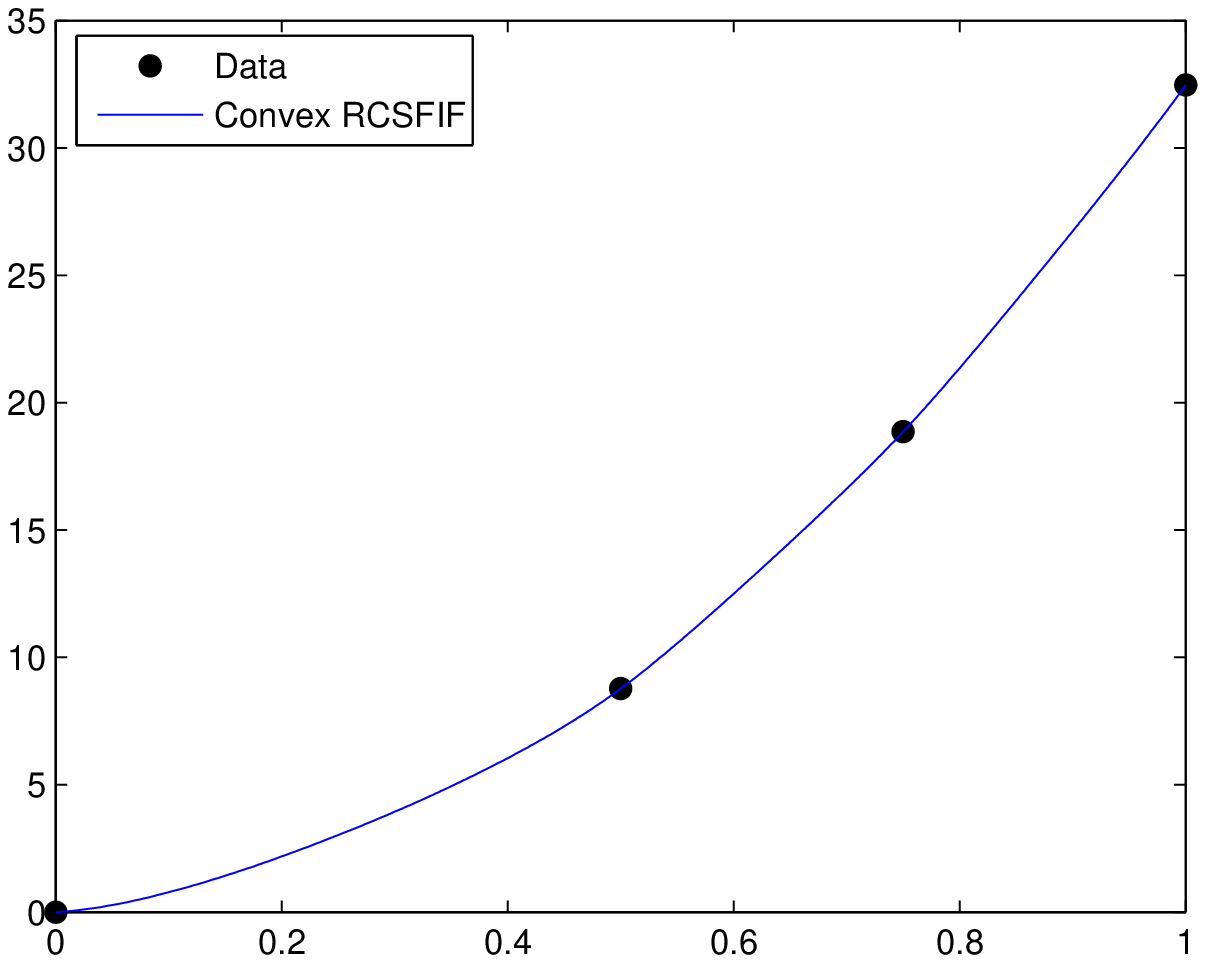,scale=0.25} \centering{\scriptsize{(f)
Convex RCSFIF}}
\end{minipage}\hfill\\
\caption{Convex and non-convex RCSFIFs for the data set in Table
\ref{chapter3table3}.}\label{chapter3fig4}
\end{center}
\end{figure}

\section{Conclusion}\label{HALDIASANGEETAsec7}
In this paper, we have constructed RCSFIF with two family of shape parameters. We identify scaling factors and shape parameters so that the graph of the corresponding RCSFIF possesses monotonicity and convexity. The scaling parameters and shape parameters play an important role in determining the shape of a RCSFIF. Thus, according to the need of an experiment for simulating objects with smooth geometrical shapes, a large flexibility in the choice of a suitable interpolating smooth fractal interpolant is offered by our approach. As in the case of vast applications of classical rational interpolants in CAM, CAD, and other mathematical, engineering applications, it is felt that RCSFIFs can find rich applications in some of these areas. Further, as classical piecewise cubic Hermite interpolant, $\mathcal{C}^1$-cubic Hermite FIF \cite{CV2}, and $\mathcal{C}^1$-rational cubic spline \cite{SHH} are special cases of RCSFIFs. It is possible to use RCSFIFs for mathematical and engineering problems where these 
approaches does not work satisfactorily. The upper bound for the error between the original function $\Phi\in \mathcal{C}^3$ and the RCCHFIF $g$ is deduced.

\end{document}